\documentclass[11pt,reqno]{amsart}

\usepackage{amsmath}
\usepackage{amsfonts}
\usepackage{mathtools}
\usepackage{amssymb}
\usepackage{amsthm}
\usepackage{tikz-cd}
\usepackage[a4paper,
            bindingoffset=0.2in,
            left=1in,
            right=1in,
            top=1in,
            bottom=1in,
            footskip=.25in]{geometry}
\allowdisplaybreaks

\newtheorem{thm}{Theorem}
\newtheorem{lemma}[thm]{Lemma}
\newtheorem{cor}[thm]{Corollary}
\newtheorem{df}[thm]{Definition}
\newtheorem{ex}[thm]{Example}
\theoremstyle{definition}

\newtheorem{rmk}[thm]{Remark}

\newcommand{\Z}{\mathbb{Z}}
\newcommand{\Q}{\mathbb{Q}}
\newcommand{\R}{\mathbb{R}}
\newcommand{\C}{\mathbb{C}}
\newcommand{\N}{\mathbb{N}}
\newcommand{\Dif}{\mathfrak{D}}
\newcommand{\M}{\mathcal{M}}
\newcommand{\T}{\mathcal{T}}
\newcommand{\sig}{\sigma}
\newcommand{\ve}{\varepsilon}
\renewcommand{\i}{\infty}
\newcommand{\GL}{\mbox{GL}_{2}}
\newcommand{\SL}{\mbox{SL}_{2}}

\renewcommand{\t}{\tau}
\renewcommand{\H}{\mathbb{H}}
\newcommand{\D}{\Delta}
\newcommand{\G}{\Gamma}
\newcommand{\g}{\gamma}
\renewcommand{\c}{\mathcal{C}}

\numberwithin{equation}{section}
\numberwithin{thm}{section}

\begin{document}

\title[Multiplicative Hecke operators]{Multiplicative Hecke operators and their applications}

\author[Chang Heon Kim and Gyucheol Shin]{Chang Heon Kim and Gyucheol Shin$^{*}$}

\address{Department of Mathematics, Sungkyunkwan University, Suwon 16419, Korea}
\email{chhkim@skku.edu}
\address{Department of Mathematics, Sungkyunkwan University, Suwon 16419, Korea}
\email{sgc7982@gmail.com}

\begin{abstract}
In this paper, we define the multiplicative Hecke operators $\T(n)$ for any positive integer on the integral weight meromorphic modular forms for $\G_{0}(N)$. We then show that they have properties similar to those of additive Hecke operators. Moreover, we prove that multiplicative Hecke eigenforms with integer Fourier coefficients are eta quotients, and vice versa. In addition, we prove that the Borcherds product and logarithmic derivative are Hecke equivariant with the multiplicative Hecke operators and the Hecke operators on the half-integral weight harmonic weak Maass forms and weight 2 meromorphic modular forms.
\end{abstract}

\maketitle

\renewcommand{\thefootnote}%
             {}
 {\footnotetext{

This work was supported by the National Research Foundation of Korea(NRF) grant funded by the Korea government(MSIT)(RS-2024-00348504).}

\section{Introduction and statement of results} 
Let $M_{1/2}^{+}(\G_{0}(4))$ denote the Kohnen plus space of the weakly holomorphic modular forms of weight 1/2 on $\G_{0}(4)$. That is, if $f(\t)\in M_{1/2}^{+}(\G_{0}(4))$, then its Fourier expansion is of the form
\begin{equation*}
f(\t)=\sum\limits_{n\equiv0,1\pmod{4}}a(n)q^{n}.
\end{equation*}
Let $H(-n)$ denote the usual Hurwitz class numbers of discriminant $-n$ and  $\mathcal{H}^{+}(\t)$ denote by
\begin{equation*}
\mathcal{H}^{+}(\t):=-\frac{1}{12}+\sum\limits_{1<n\equiv0,3\pmod{4}}H(-n)q^{n}.
\end{equation*}
For each $f(\t)=\sum_{n=n_{0}}^{\i}a(n)q^{n}\in M_{1/2}^{+}(\G_{0}(4))$, the map $B$ is defined by
\begin{equation*}
B(f(\t)):=q^{-h}\prod\limits_{n=1}^{\i}(1-q^{n})^{a(n^{2})}
\end{equation*}
where $h$ is a constant term of $f(\t)\mathcal{H}^{+}(\t)$. 
In \cite[Theorem 14.1]{Bor}, Borcherds proved that the exponents of the infinite product expansions of integral weight meromorphic modular forms for certain character of $\SL(\Z)$ with integer coefficients and leading coefficient 1 with a Heegner divisor(its zeros and poles are supported at the cusp at $i\i$ and CM points) are the Fourier coefficients of weakly holomorphic modular forms of weight 1/2 on $\G_{0}(4)$ that satisfy the Kohnen plus condition.
More precisely, let $\M^{H}(1)$ denote the (multiplicative) group of the integral weight meromorphic modular forms for some character of $\SL(\Z)$ with integer coefficients, leading coeffficient 1, and a Heegner divisor. Borcherds proved that given $f(\t)\in M_{1/2}^{+}(\G_{0}(4))$ with integer Fourier coefficients, the function $B(f(\t))\in\M^{H}(1)$. Furthermore, $B$ is an isomorphism
\begin{equation*}
B : M_{\frac{1}{2}}^{+}(\G_{0}(4))\rightarrow \M^{H}(1).
\end{equation*}
The weight of $B(f)$ is $a(0)$, and the multiplicity of the zero of $B(f)$ at the CM point of discriminant $D<0$ is
\begin{equation*}
\sum\limits_{n>0}a(Dn^{2}).
\end{equation*}
Borcherds raised several open questions regarding this isomorphism in \cite[17.10]{Bor}.
\begin{enumerate}
\item Extend the isomorphism to higher levels.
\item Find some action of a Hecke algebra that commutes with the isomorphism.
\end{enumerate}

\indent The former was partially answered by Borcherds himself in \cite{Borc}, in which he simplified the proofs of his earlier results and extended them for both weight and level. Bruiner and Ono \cite{BO} also established a generalized Borcherds lift for harmonic weak Maass forms: They showed that, if $f$ is a vector-valued harmonic weak Maass form of weight $k$ and type $\rho_{L}$, then there exists a meromorphic modular form of $\G_{0}(N)$ with a unitary character $\sigma$ satisfying several conditions (for more details, see \cite[Theorem 6.1]{BO}).
\\
\indent The latter question was partially answered by Guerzhoy. Let $\M(N)$ be the multiplicative group of meromorphic modular forms for $\G_{0}(N)$ with a unitary multiplier system, integer Fourier coefficients, and leading coefficient 1. More precisely, we mean $f:\H\rightarrow\C$ is a meromorphic modular form of weight $k$ for $\G_{0}(N)$ with a unitary multiplier system $\chi:\G_{0}(N)\rightarrow S^{1}$, if it transforms as
\begin{equation*}
f(\g\t)=\chi(\g)(c\t+d)^{k}f(\t)
\end{equation*}
for all $\t\in\H$ and $\g=\begin{psmallmatrix}a&b\\c&d\end{psmallmatrix}\in\G_{0}(N)$, and is meromorphic at the cusps. We let $\M_{k,h}(N)\subset\M(N)$ denote the subset that consists of modular forms of weight $k$ for which the order of the zero at $i\i$ is $h$ if $h\geq0$ or the pole at $i\i$ is $-h$ if $h<0$. In \cite{Guer}, the multiplicative Hecke operator $\T(p)$ for prime $p$ acting on $\M(N)$ is given by
\begin{equation*}
f|\T(p):=\ve p^{k(p-1)/2}\prod_{\substack{ad=p\\0\leq b<d}}f|_{k}\begin{pmatrix}a&b\\0&d\end{pmatrix},
\end{equation*}
where $\ve$ is a constant chosen such that the leading coefficient of $f|\T(p)$ is 1. He then proved that the Borcherds isomorphism is Hecke equivariant if one considers a multiplicative Hecke operator $\T(p)$ and the usual Hecke operator $pT_{1/2}(p^2)$ acting on the half-integral weight modular form on $\G_{0}(4)$. Recently, Jeon, Kang, and the first author \cite{JKK} have extended this result to higher levels by considering the generalized Borcherds product defined in \cite[Theorem 6.1]{BO}. They showed that the following diagram is commutative.
\begin{equation*}
\begin{tikzcd}
H_{\frac{1}{2},\tilde{\rho}_{N}}'\arrow{r}{B}\arrow{d}{pT_{\frac{1}{2}}(p^{2})} & \M^{H}(N)\arrow{d}{\T(p)}
\\
H_{\frac{1}{2},\tilde{\rho}_{N}}'\arrow{r}{B} & \M^{H}(N)
\end{tikzcd}
\end{equation*}
Here, $p$ is a prime that does not divide $N$ and, by abusing the notation, $B$ denotes the generalized Borcherds product defined in \cite[Theorem 6.1]{BO}. We refer to \cite{BO} and \cite{JKK} for the definitions of $H_{\frac{1}{2},\tilde{\rho}_{N}}'$ and $\M^{H}(N)$ respectively.
\\
\indent
On the other hand, in \cite{BKO}, Bruinier, Kohnen, and Ono showed that 
\begin{equation}\label{BKOdivisor}
-\frac{\partial_{k}(f)}{f}=\sum\limits_{z\in\mathcal{F}}e_{z}{\rm ord}_{z}(f)H_{z}(\t)
\end{equation}
where $H_{z}(\t):=\sum_{n=0}^{\i}J_{n}(z)q^{n}$, $z,\t\in\H$, $q=e^{2\pi i\t}$, $J_{n}=(j-744)|T(n)$, $\mathcal{F}$ is the fundamental domain for $\SL(\Z)$, $e_{z}:=1/|\SL(\Z)_{z}/\{\pm 1\}|$, $\SL(\Z)_{z}$ is the stabilizer of $z$, $\partial_{k}=\Theta(f)-kE_{2}f/12$ is the Serre derivative, and $\Theta(f):=q\frac{df}{dq}$. These results were generalized by Bringmann et al. \cite{BKLOR} and Choi, Lee, and Lim \cite{CLL} to Niebur-Poincar\'e harmonic weak Maass functions of arbitrary
level $N$. They proved that the logarithmic derivative of a meromorphic modular form for $\G_{0}(N)$ is explicitly described in terms of the values of the Niebur-Poincar\'e series at its divisors in $\H$. In \cite{JKK}, the authors showed that the logarithmic derivative defined on the multiplicative group of meromorphic modular forms of $\G_{0}(N)$ with a unitary multiplier system is also Hecke equivariant under the multiplicative Hecke operator $\T(p)$ and the additive Hecke operator $T(p)$ on the weight 2 meromorphic modular forms of $\G_{0}(N)$.
\\
\indent These results were proved only for a prime number $p$. In additive Hecke algebra, acting on the space of modular forms of weight $k$ for $\G_{0}(N)$, $\{T(p):p \text{ is a prime}\}$ forms a building block via \eqref{additiverelation}. However, to the best of our knowledge, there is no known definition of $\T(n)$ for all positive integers $n$ or the results related to the multiplicative Hecke operators corresponding to \eqref{additiverelation}. Therefore, Guerzhoy’s definition of $\T(p)$ cannot be directly extended to all positive integers $n$ or even to the power of the prime $p^{r}$. In this paper, we define multiplicative Hecke operators for all positive integers $n$ and show that $\T(n)$ is generated by $\T(p)'s$, where $p|n$, in Theorem \ref{Heckealgebra}. To achieve this, we first define multiplicative Hecke operators for the power of a prime.
\begin{df}\label{1.1}
Let $p$ be a prime and let $r,N$ be positive integers. Let $f\in\M(N)$. Then, the multiplicative Hecke operator for the power of the prime acting on $\M(N)$ is defined by
\begin{equation*}
f|\T(p^{r}):=
\begin{cases}
\ve p^{\frac{k}{1-p}\big(\frac{p^{r+1}-p}{p-1}-\frac{r}{2}(p^{r+1}+1)\big)}\prod\limits_{\substack{ad=p^{r}\\ 0\leq b<d}}f|_{k}\begin{pmatrix}a&b\\0&d\end{pmatrix}, & \text{ if } p\nmid N,
\\
\ve p^{\frac{rkp^{r}}{2}}\prod\limits_{j=0}^{p^{r}-1}f|_{k}\begin{pmatrix}1&j\\0&p^{r}\end{pmatrix}, & \text{ if } p|N,
\end{cases}
\end{equation*}
where $\ve$ is a constant chosen such that the leading coefficient of $f|\T(p^{r})$ is equal to 1. We also denote $\T(1)$ as the trivial action on $\M(N)$. We note that the definition of $\T(p^{r})$ depends on level $N$ and weight $k$. However, for simplicity, we omit this dependence from the notation.
\end{df}

Once we ignore the normalizing factor in the definition of $\T(p^{r})$, it becomes very similar to additive Hecke operators. This is why the definition depends on whether $p$ divides $N$. At first glance, the power of $p$ looks weird. In fact, following Guerzhoy's definition, the power of $p$ is chosen to normalize the norm of $\varepsilon$ to 1. Note that if $r=1$, then the two definitions (Guerzhoy's and ours) agree. However, there is a slight difference between them; we exclude the (twisted) Heegner divisor condition as it is not essential when defining multiplicative Hecke operators. There are $\sigma(p^{r})$(resp. $p^{r}$) products when $p\nmid N$(resp. $p|N$), where $\sigma(n)=\sum_{0<d|n}d$ is the sum of distinct divisor of $n$. Thus, the modular form of the weight $k$ is lifted by $\T(p^{r})$ into a modular form of the weight $k\sigma(p^{r})$(resp. $kp^{r}$). Moreover, multiplicative Hecke operators preserve the multiplicative group $\M(N)$. This observation leads to our first result.

\begin{thm}\label{thmctprn}
Let $p$ be a prime and $r$ be a positive integer. Let $f(\t)\in\M_{k,h}(N)\subset\M(N)$ be a weight $k$ meromorphic modular form for $\G_{0}(N)$ of the form
\begin{equation}\label{infiniteproduct}
f(\t)=q^{h}\prod\limits_{n=1}^{\i}(1-q^{n})^{c(n)}.
\end{equation}
Then the following are true:
\begin{enumerate}
\item 
\begin{equation}\label{weightlift}
f|\T(p^{r})\in
\begin{cases}
\M_{k\sigma(p^{r}),h\sigma(p^{r})}(N) & \text{ if } p\nmid N,
\\
\M_{kp^{r},h}(N) & \text{ if } p|N.
\end{cases}
\end{equation}
\item The multiplicative Hecke operator $\T(p^{r})$ acting on $\M(N)$ is given by 
\begin{equation*}
f|\T(p^{r})=
\begin{cases}
q^{h\sigma(p^{r})}\prod\limits_{n=1}^{\i}(1-q^{n})^{c_{p^{r}}(n)} & \text{ if } p\nmid N,
\\
q^{h}\prod\limits_{n=1}^{\i}(1-q^{n})^{c_{p^{r}}(n)} & \text{ if } p|N,
\end{cases}
\end{equation*}
where
\begin{equation}\label{ctpr}
c_{p^{r}}(n)=
\begin{cases}
\sum\limits_{i=0}^{r}p^{i}\c\Big(p^{i},\frac{n}{p^{r-i}}\Big)+\sum\limits_{k=0}^{r-1}\sum\limits_{i=0}^{k}\chi_{p}\Big(\frac{n}{p^{r-k-1}}\Big)p^{i}\c\Big(p^{i},\frac{n}{p^{r-k-1}}\Big) & \text{ if } p\nmid N,
\\ 
p^{r}\c(p^{r},n)+\sum\limits_{i=0}^{r-1}\chi_{p}(n)p^{i}\c(p^{i},n) & \text{ if } p|N.
\end{cases}
\end{equation}
Here, $\c(*,*)$ is defined by
\begin{equation*}
\c\bigg(X,\frac{Y}{Z}\bigg):=
\begin{cases}
c\big(\frac{XY}{Z}\big) & \text{ if } Z|Y
\\
0 & \text{ if } Z\nmid Y
\end{cases}
\end{equation*}
where $X,Y,Z\in\Z$ with $Z\ne0$ and $\chi_{p}$ is the trivial Dirichlet character modulo $p$.

\end{enumerate}
\end{thm}

\begin{rmk}
In \cite[Proposition 2.1]{BKO}, it was shown that, if $f(\t)=q^{h}(1+\sum_{n=1}^{\i}a(n)q^{n})$ is a meromorphic modular form of weight $k$, then there exist complex numbers $c(n)$ such that $f(\t)$ is written as \eqref{infiniteproduct}. Furthermore, \cite[Proposition 5.1]{JKK} proved that $c(n)$ has the following recursion formula:
\begin{equation}\label{recursion}
c(n)=-a(n)-\frac{1}{n}\bigg(\sum_{\substack{1\leq u<n\\u|n}}uc(u)+\sum\limits_{1\leq s<n}a(n-s)\sum\limits_{u|s}uc(u)\bigg).
\end{equation}
This implies that if $f(\t)\in\M(N)$, then $c(n)\in\Q$. Moreover, one can see that $c(n)\in\Z$ for all $n\geq1$. To be more precise, suppose that there exists $n\in\N$ such that $c(n)\not\in\Z$. Denote $n_{0}:=\inf\{n\in\N:c(n)\not\in\Z\}$ and then consider $a(n_{0})$. Using the binomial theorem and considering the product $\prod_{n\leq n_{0}}(1-q^{n})^{c(n)}$, one can see that $a(n_{0})=-c(n_{0})+(\text{some integer})$ since the product $\prod_{n>n_{0}}(1-q^{n})^{c(n)}$ does not contribute to $a(n_{0})$, which means that $a(n_{0})$ is also not an integer. This is a contradiction.
\end{rmk}

\begin{ex}\label{multiplier}
\begin{enumerate}
\item Let $E_{4}(\t)$ be the Eisenstein series of weight 4 on $\SL(\Z)$. In Example \ref{1.16}, we show 
\begin{equation*}
E_{4}|\T(3)=E_{4}(\t)\Delta(\t)(j(\t)-j(3\rho)),
\end{equation*}
where $\rho=e^{2\pi i/3}$. It follows that the multiplier system $\chi_{triv}$ of $E_{4}$ is preserved by $\T(3)$. 
\item 
Suppose that $f(\t)$ is the eta quotient given by
\begin{equation*}
f(\t)=\frac{\eta(5\t)^{5}}{\eta(\t)}\in M_{2}\bigg(\G_{0}(5),\Big(\frac{\cdot}{5}\Big)\bigg).
\end{equation*}
We show in Theorem \ref{eigenform} that
\begin{equation*}
f(\t)|\T(3)=f^{4}(\t),
\end{equation*}
which implies $f(\t)|\T(3)\in M_{8}(\G_{0}(5))$. Thus, in this case, the multiplicative Hecke operators do not preserve the multiplier system equipped by the original ones.
\end{enumerate}
\end{ex}

\begin{cor}\label{congruence}
Let $N$ be a positive integer and $p$ be a prime that does not divide $N$. Let $r$ be a positive integer. Suppose that 
\begin{equation*}
f(\t)=q^{h}\prod\limits_{n=1}^{\i}(1-q^{n})^{c(n)}\in\M_{k,h}(N).
\end{equation*}
We write $n=p^{m}e$ for some nonnegative integers $m$ and $e$ with $p\nmid e$. Then, we have 
\begin{equation*}
c_{p^{r}}(n)\equiv 
\begin{cases}
c(e)\pmod{p} & \text{ if } 0\leq m< r,
\\
c(\frac{n}{p^{r}})\pmod{p} & \text{ otherwise.}
\end{cases}
\end{equation*}
In particular, if $(n,p)=1$, we have $c_{p^{r}}(n)\equiv c(n)\pmod{p}$.
\end{cor}

We now define the multiplicative Hecke operator for all positive integers $n$. In Section 3, we prove that multiplicative Hecke operators commute and establish an algebraic structure between multiplicative Hecke operators for the power of primes. This leads to Definition \ref{deftn} and Theorem \ref{Heckealgebra}.

\begin{df}\label{deftn}
Let $n=\prod_{p_{i}|n}p_{i}^{r_{i}}$ be a positive integer. Then, the multiplicative Hecke operator $\T(n)$ acting on $\M(N)$ is defined by
\begin{equation*}
\T(n):=\prod\limits_{p_{i}|n}\T(p_{i}^{r_{i}}).
\end{equation*}
Like the definition of the multiplicative Hecke operator for the power of the prime, we omit level $N$ and weight $k$ from the notation.
\end{df}

\begin{thm}\label{Heckealgebra}
Let $m,n$ be positive integers and $f(\t)\in\M(N)$. Then, we have
\begin{equation}\label{mn}
f|\T(m)\T(n)=f|\T(n)\T(m)=\prod_{0<d|(m,n)}\bigg(\chi_{N}(d)f|\T\bigg(\frac{mn}{d^{2}}\bigg)\bigg)^{d}.
\end{equation}
\end{thm}

Ramanujan's tau function $\t(n)$ is given by the Fourier coefficients of the Delta function (the unique normalized cusp form of weight 12 on $\SL(\Z)$). Namely,
\begin{equation*}
\Delta(\t)=\sum\limits_{n=1}^{\i}\t(n)q^{n}.
\end{equation*}
The multiplicativity of Ramanujan's tau function $\t(n)$ can be easily proven if one consider the fact that the Delta function $\D(\t)$ is an additive Hecke eigenform. In contrast, multiplicative Hecke operators are associated with the sum of divisors function. By employing multiplicative Hecke operators, we obtain a new proof for the following identity.

\begin{cor}\label{main3}
Let $m,n$ be positive integers. Then
\begin{equation*}
\sig(m)\sig(n)=\sum\limits_{0<d|(m,n)}d\sig\bigg(\frac{mn}{d^{2}}\bigg)
\end{equation*}
holds.
\end{cor}

\begin{rmk}
There are alternative proofs of this theorem. See \cite[Corollary 10.4.6]{CS} and \cite[Chapter 2, Exercise 32]{Apo} for further details.
\end{rmk}

We show that multiplicative Hecke operators commute in Theorem \ref{Heckealgebra}. This indicates that we can consider the eigenforms on $\G_{0}(N)$ for multiplicative Hecke operators. Like the additive Hecke eigenform, we define a new type of Hecke eigenform by employing multiplicative Hecke operators.

\begin{df}
We call $f(\t)\in\M(N)$ a multiplicative Hecke eigenform of $\G_{0}(N)$ if, for every prime $p$ not dividing $N$, there exist complex numbers $\lambda(p)$ such that
\begin{equation}\label{eigenvalue}
f(\t)|\T(p)=f(\t)^{\lambda(p)}.
\end{equation}
\end{df}

\begin{rmk}
\begin{enumerate}
\item As mentioned earlier, we have $c(n)\in\Z$, which means that $c_{p}(n)$ is also an integer for each $n\geq1$. Therefore, the multiplicative eigenvalues $\lambda(p)$ are actually integers. 
\item Unlike the eigenvalues of additive Hecke operators, the eigenvalues of multiplicative Hecke operators are constrained. More precisely, suppose that at least one of $k$ and $h$ is nonzero and $f(\t)\in\M_{k,h}(N)$ is a multiplicative Hecke eigenform. Then, we have $f|\T(p)\in\M_{k\sigma(p),h\sigma(p)}(N)$. This means that, if $\lambda(p)$ is not equal to $\sigma(p)$ for some prime $p$ that does not divide $N$, then the weight or vanishing order at $i\i$ on both sides of \eqref{eigenvalue} is not the same. Furthermore, even if $k=h=0$, it should be $\lambda(p)=\sigma(p)$ for all primes $p$ not dividing $N$. In this case, $\lambda(p)=\sigma(p)$ is obtained by comparing the number of zeros or poles with the multiplicity of both sides of \eqref{eigenvalue}.
\item By Theorem \ref{Heckealgebra}, if \eqref{eigenvalue} holds, then such an eigenform is actually an eigenform of all $\T(n)$ for $n\geq1$ relatively prime to $N$.
\item If we allow the condition that $f|\T(p)=f(\t)^{\lambda(p)}$ when $p|N$, then there are only a few eigenforms satisfying this new condition. For example, if $N$ is a prime, then the multiplicative Hecke eigenform satisfying $f|\T(p)=f(\t)^{\lambda(p)}$ for all $p$, including $N$, is of the form
\begin{equation*}
f(\t)=\frac{\eta(\t)^{pt}}{\eta(p\t)^{t}}.
\end{equation*}
This can be easily seen by \eqref{coe}.
\end{enumerate}
\end{rmk}

\indent An eta quotient of level $N$ is a function of the form
\begin{equation*}
f(\t):=\prod\limits_{\delta|N}\eta(\delta\t)^{r_{\delta}}
\end{equation*}
where $\eta(\t)$ is the Dedekind eta function and $r_{\delta}\in\Z$. Since the Dedekind eta function is holomorphic and has no zeros on $\H$, the eta quotient does not vanish on $\H$. In \cite{New}, Newman proved that if
\begin{equation*}
\sum\limits_{\delta|N}\delta r_{\delta}\equiv0\pmod{24}, \;\; \sum\limits_{\delta|N}\frac{N}{\delta}r_{\delta}\equiv0\pmod{24}
\end{equation*}
then $f(\t)$ transforms like a modular form of weight $k=\frac{1}{2}\sum_{\delta|N}r_{\delta}$ for $\G_{0}(N)$ withe some multiplier system $\chi$. Here $\chi$ is defined by $\chi(\g):=\big(\frac{(-1)^{k}s}{d}\big)$ where $\g=\begin{psmallmatrix}a&b\\c&d\end{psmallmatrix}\in\G_{0}(N)$ and $s:=\prod_{\delta|N}\delta^{r_{\delta}}$. Throughout this paper, whenever we refer to an eta quotient of level $N$, we assume that it always satisfies the above conditions so that it is a weakly holomorphic modular form on $\G_{0}(N)$.

\begin{ex}
Let
\begin{equation*}
f(\t):=-\frac{1}{24}(E_{2}(\t)-3E_{2}(2\t)+2E_{2}(4\t))=\sum\limits_{n=0}^{\i}\sigma(2n+1)q^{2n+1}\in M_{2}(\G_{0}(4)).
\end{equation*}
From \cite[Chapter 3, \S 3, Problem 10]{Kob}, we have $f(\t)=\eta(4\t)^{8}/\eta(2\t)^{4}$ and write $f(\t)=q\prod_{n=1}^{\i}(1-q^{n})^{c(n)}$. Then for each positive integer $n$, we have 
\begin{equation*}
c(n)=\sum\limits_{\delta|(n,4)}r_{\delta},
\end{equation*}
where $r_{1}=0$, $r_{2}=-4$, and $r_{4}=8$. Let $p$ be an odd prime. Then, we observe that $c(pn)=c(n)$ for all $n$. Since $c_{p}(n)=pc(pn)+\chi_{p}(n)c(n)+c(n/p)$ by \eqref{ctpr}, we obtain $c_{p}(n)=(p+1)c(n)$. Thus, $f|\T(p)=f(\t)^{\sigma(p)}$. 
\end{ex}

\indent The $n$th exponents in the infinite product expansion of the eta quotient can be easily computed and have a simple formula. For example, if $n$ is a positive integer coprime to $N$ and $f(\t)$ is an eta quotient, then $c(n)=c(1)$($c(n)$ is defined in \eqref{infiniteproduct}). We prove that the exponents in the infinite product expansion of the multiplicative Hecke eigenform obey relations similar to those of the eta quotient in Section 4. Based on this observation, we obtain the following result. 

\begin{thm}\label{eigenform}
Let $N$ be a positive integer and $\{r_{\delta}\}$ be the set of integers. Suppose that 
\begin{equation*}
f(\t):=\prod\limits_{0<\delta|N}\eta(\delta\t)^{r_{\delta}}
\end{equation*}
is an eta quotient of level $N$. Then, $f(\t)$ is the multiplicative Hecke eigenform of level $N$.
\end{thm}

In \cite{Koh}, Kohnen proved that for a meromorphic modular form $f$ of weight $k$ and squarefree level $N$, the following are equivalent:
\begin{enumerate}
\item $f$ has no zeros or poles on $\H$.
\item the $n$th exponent in the infinite product expansion of $f$ defined in \eqref{infiniteproduct} depend only on $(n,N)$.
\end{enumerate}

\indent A direct consequence of his proof is that, if $N$ is squarefree and $f(\t)$ is a weakly holomorphic modular form of weight $k$ on $\G_{0}(N)$, that has no zeros on $\H$, then there exists some positive integer $t$ such that $f(\t)^{t}$ is an eta quotient up to scalar multiple. In \cite{RW}, Rouse and Webb proved that if $f$ is a weakly holomorphic modular form of weight $k$ for $\G_{0}(N)$ with integer Fourier coefficients having no zeros or poles on $\H$, then $f(\t)$ is the eta quotient up to scalar multiple. From these results, we characterize the multiplicative Hecke eigenform of level $N$.

\begin{thm}\label{equivalent}
Let $f(\t)\in\M(N)$. Then the following are equivalent:
\begin{enumerate}
\item $f(\t)$ is a multiplicative Hecke eigenform.
\item $f(\t)$ has no zeros or poles on $\H$.
\item $f(\t)$ is an eta quotient.
\end{enumerate}
\end{thm}

Now, as an application of multiplicative Hecke operators, we show that generalized Borcherds product is Hecke equivariant, which answers the question posed by Borcherds \cite[17.10]{Bor} completely.

\begin{thm}\label{main1}
Let $\Delta$ be a fundamental discriminant. Let $n$ be a positive integer such that $(n,N)=(n,\Delta)=1$. Let $H_{k,\tilde{\rho}_{N}}'$ be the additive subgroup of $H_{k,\tilde{\rho}_{N}}$ defined in \cite[Theorem 3.1]{JKK}. Then the following diagram of the groups are commutative.
\begin{equation*}
\begin{tikzcd}
H_{\frac{1}{2},\tilde{\rho}_{N}}'\arrow{r}{B}\arrow{d}{nT_{\frac{1}{2}}(n^{2})} & \M^{H}(N)\arrow{d}{\T(n)}
\\
H_{\frac{1}{2},\tilde{\rho}_{N}}'\arrow{r}{B} & \M^{H}(N)
\end{tikzcd}
\end{equation*}
\end{thm}

We provide some examples associated to Theorem \ref{main1}. Let $f_{d}(\t)\in M_{1/2}^{+}(4)$ be a unique modular form with a Fourier expansion of 
\begin{equation*}
f_{d}(\t)=q^{-d}+\sum_{\substack{D>0\\ D\equiv0,1\pmod{4}}}A(D,d)q^{D}.
\end{equation*}
Then, $\{f_{d}(\t):0\leq d\equiv0,3\pmod{4}\}$ form a basis for $M_{1/2}^{+}(4)$. Note that $f_{d}$ can be interpreted as an element in $H_{1/2,\tilde{\rho}_{1}}'$ via the isomorphism given in \cite[Theorem 1]{Cho}.

\begin{ex}\label{1.16}
Let $E_{k}(\t)$ be the Eisenstein series of weight $k$ on $\SL(\Z)$. The Fourier expansion and the infinite product expansion of $E_{4}(\t)$ are given by
\begin{align*}
E_{4}(\t)&=1+240q+2160q^{2}+6720q^{3}+17520q^{4}+\cdots
\\
&=\prod\limits_{n=1}^{\i}(1-q^{n})^{c(n)}=(1-q)^{-240}(1-q^{2})^{26760}(1-q^{3})^{-4096240}(1-q^{4})^{708938760}\cdots.
\end{align*}
From the Borcherds isomorphism for $\SL(\Z)$, we obtain $E_{4}(\t)=B(f_{3}+4f_{0})$. Furthermore, by Theorem \ref{main1}, we have 
\begin{equation*}
E_{4}|\T(3)=B((f_{3}+4f_{0})|3T_{1/2}(9))=B(f_{27}+16f_{0})=E_{4}(\t)\Delta(\t)(j(\t)-j(3\rho))
\end{equation*}
where $\rho=e^{2\pi i/3}$ and $j(\t)=q^{-1}+744+O(q)$ is the classical $j$-invariant. It follows that 
\begin{equation*}
E_{4}|\T(3)\T(3)=E_{4}(\t)\Delta(\t)^{5}j(\t)(j(\t)-j(3\rho))(j(\t)-j(9\rho))\Big(j(\t)-j\big(\frac{3\rho+1}{3}\big)\Big)\Big(j(\t)-j\big(\frac{3\rho+2}{3}\big)\Big).
\end{equation*}
Then by Theorem \ref{Heckealgebra}, we have
\begin{equation*}
E_{4}|\T(9)=E_{4}(\t)\Delta(\t)^{4}(j(\t)-j(3\rho))(j(\t)-j(9\rho))\Big(j(\t)-j\big(\frac{3\rho+1}{3}\big)\Big)\Big(j(\t)-j\big(\frac{3\rho+2}{3}\big)\Big).
\end{equation*}
\end{ex}

\indent The multiplicative Hecke operators are associated with the constant term of a modular polynomial (as a one-variable polynomial). More precisely, let $m$ be a positive integer and $z\in\H$. The $m$th modular polynomial is defined as 
\begin{equation*}
\Phi_{m}(X,j(z)):=\prod\limits_{\g\in\G_{0}(1)\backslash A^{m}}(X-j(\g z)),
\end{equation*}
where $A^{m}$ denotes the set of $2\times2$ matrices with integer entries whose determinant equals $m$. It is well known that $\Phi_{m}$ is a polynomial of $j(z)$. Thus, there exists a polynomial $\Psi_{m}(X,Y)\in\Z[X,Y]$ such that $\Psi_{m}(X,j(z))=\Phi_{m}(X,j(z))$. By abusing this notation, we set $\Phi_{m}(X,Y):=\Psi_{m}(X,Y)$. It is well known that $\Phi_{m}(X,Y)$ is symmetric up to the sign and of degree $\sigma(m)$ as a polynomial in $X$. Furthermore, if $m$ is squarefree, then $\Phi_{m}$ is irreducible. For further details, refer to \cite [Chapter 6.1]{Zag} and \cite[Chapter 11]{Cox}. By the definition of multiplicative Hecke operators, $\Phi_{m}(0,j(\t))$ is equal to $j|\T(m)$ up to the sign.
\\
\indent For example, let $m=3$. The modular polynomial $\Phi_{3}(X,Y)$ is given by
\begin{align*}
\Phi_{3}(X,Y)&=X(X+2^{15}\cdot3\cdot5^3)^{3}+Y(Y+2^{15}\cdot3\cdot5^{3})^{3}+2^{3}\cdot3^{2}\cdot31X^{2}Y^{2}(X+Y)
\\
&-X^{3}Y^{3}-2^{2}\cdot3^{3}\cdot9907XY(X^{2}+Y^{2})+2\cdot3^{4}\cdot13\cdot193\cdot6367X^{2}Y^{2}
\\
&+2^{16}\cdot3^{5}\cdot5^{3}\cdot17\cdot263XY(X+Y)-2^{31}\cdot5^{6}\cdot22973XY.
\end{align*}
Note that
\begin{equation}\label{jt}
j|\T(3)=\Phi_{3}(0,j(\t))=j(\t)(j(\t)+2^{15}\cdot3\cdot5^{3})^{3}=j(\t)(j(\t)-j(3\rho))^{3}.
\end{equation}

\begin{ex}
We provide a different approach to \eqref{jt} by using Theorem \ref{main1}, which states that
\begin{equation*}
j|\T(3)=B(3f_{3}|3T_{\frac{1}{2}}(9))
\end{equation*}
since $j(\t)=B(3f_{3})$ (see \cite[Theorem 3]{Zagier}). One can easily compute $f_{3}|3T_{1/2}(9)=f_{27}$ and $B(f_{27})=j^{1/3}(\t)(j(\t)-j(3\rho))$. Hence, \eqref{jt} is obtained from the commutative diagram of Theorem \ref{main1}.

\end{ex}

We denote the space of meromorphic modular forms of weight 2 on $\G_{0}(N)$ by $M_{2}^{mero}(N)$ and define a map $\Dif$ from $\M(N)$ to $M_{2}^{mero}(N)$ by
\begin{equation*}
\begin{split}
&\Dif:\M(N)\rightarrow M_{2}^{mero}(N)
\\
&f\mapsto \Dif(f):=\frac{\Theta(f)}{f}-\frac{kE_{2}(\t)}{12}.
\end{split}
\end{equation*}
We show that the map $\Dif$ also commutes with the action of Hecke operators $\T(n)$ and $T_{2}(n)$.

\begin{thm}\label{Dequivn}
Let $n,N$ be positive integers such that $(n,N)=1$. Then the following diagram is commutative.
\begin{equation*}
\begin{tikzcd}
\M(N)\arrow{r}{\Dif}\arrow{d}{\T(n)} & M_{2}^{\text{mero}}(N)\arrow{d}{T_{2}(n)}
\\
\M(N)\arrow{r}{\Dif} & M_{2}^{\text{mero}}(N)
\end{tikzcd}
\end{equation*}
\end{thm}

\begin{rmk}
When $n$ is a prime, Jeon Kang and the first author \cite[Theorem 3.2]{JKK} proved this theorem by a generalized Borcherds product. However, we provide an alternative and much easier proof of this theorem.
\end{rmk}

The remainder of this paper is organized as follows. In Section 2, we first recall some basic properties of the additive Hecke operators. In Section 3, we prove Theorems \ref{thmctprn}, \ref{Heckealgebra}, and \ref{main3} by considering the cases separately when $p$ divides $N$ and when it does not. In Section 4, we first describe the exponents in the infinite product expansion of the multiplicative Hecke eigenform. Next, we prove Theorems \ref{eigenform} and \ref{equivalent}. Finally, in Section 5, we prove Theorem \ref{main1} and Theorem \ref{Dequivn}.

\section{Additive Hecke operators}
Throughout this paper, $(a,b)$ represents the greatest common divisor of $a$ and $b$. Moreover, $\chi_{n}$ denotes the trivial Dirichlet character modulo $n$ and $\zeta_{n}:=e^{2\pi i/n}$.
In this section, we briefly recall some basic facts regarding additive Hecke operators. 
\\
\indent From now on, to distinguish between the classical additive Hecke operators we recall below and the multiplicative Hecke operators defined in this paper, we call them additive Hecke operators and multiplicative Hecke operators respectively. 
\\
\indent Let $N$ be a positive integer, and let $k$ be a positive integer greater than 2. We denote the space of weight $k$ holomorphic modular forms of $\G_{0}(N)$ as $M_{k}(N)$. For each $f(\t)\in M_{k}(N)$, the weight $k$ slash operator $|_{k}$ on $M_{k}(N)$ is defined by
\begin{equation*}
(f|_{k}\g)(\t):=\det(\g)^{\frac{k}{2}}(c\t+d)^{-k}f\bigg(\frac{a\t+b}{c\t+d}\bigg) \text{ where } \g=\begin{pmatrix}a&b\\c&d\end{pmatrix}\in\GL^{+}(\R).
\end{equation*}
Let $p$ be a prime and $A^{p}$ be the set of $2\times2$ matrices with integer entries whose determinant equals $p$. For $f(\t)=\sum_{n=0}^{\i}a(n)q^{n}\in M_{k}(N)$, the action of the additive Hecke operator $T_{k}(p)$ on $f(\t)$ is defined as
\begin{equation*}
f|T_{k}(p):=p^{k/2-1}\sum\limits_{\alpha_{i}\in\G_{0}(N)\backslash A^{p}}f|_{k}\alpha_{i}=\sum\limits_{n=0}^{\i}\big(a(pn)+\chi_{N}(p)p^{k-1}a(n/p)\big)q^{n},
\end{equation*}
where $a(n/p)=0$ unless $p|n$. In general, if $n$ is a positive integer, then the action of the additive Hecke operator of $T_{k}(n)$ is defined by
\begin{equation*}
f|T_{k}(n):=\sum\limits_{m=0}^{\i}\bigg(\sum\limits_{0<d|(m,n)}\chi_{N}(d)d^{k-1}a\bigg(\frac{mn}{d^{2}}\bigg)\bigg)q^{m}.
\end{equation*}
Although $T_{k}(n)$ depends on the level $N$, we omit them for convenience. For the cases equipped the nontrivial multiplier systems, we refer to \cite{CS,DS,Kob}. The following are standard results for the additive Hecke operators. We refer to \cite{DS,Kob}.
\begin{thm}
The following are true:
\begin{enumerate}
\item For positive integers $m$ and $n$, the following identity holds:
\begin{equation}\label{additiverelation}
T_{k}(m)T_{k}(n)=\sum\limits_{0<d|(m,n)}\chi_{N}(d)d^{k-1}T_{k}\bigg(\frac{mn}{d^{2}}\bigg).
\end{equation}
In particular, $T_{k}(n)T_{k}(m)=T_{k}(m)T_{k}(n)=T_{k}(mn)$ if $(m,n)=1$.
\item $T_{k}(n)\in{\rm End}(M_{k}(N))$ for all $n$.
\end{enumerate}
\end{thm}

We call $f(\t)\in M_{k}(N)$ an additive Hecke eigenform if it is an eigenform of all Hecke operators $T_{k}(n)$ such that $(n,N)=1$.

If $p$ is coprime to $N$, the additive Hecke operator $T_{\frac{1}{2}}(p^{2})$ acts on spaces of modular forms of weight 1/2. More precisely, if $f\in M_{1/2}(\G_{0}(4N))$ has a Fourier expansion $f=\sum_{n=0}^{\i} a(n)q^{n}$, then the Fourier expansion of $f|T_{\frac{1}{2}}(p^{2})$ is given by
\begin{equation*}
f|T_{\frac{1}{2}}(p^{2}):=\sum\limits_{n=0}^{\i}\bigg(a(p^{2}n)+\frac{1}{p}\Big(\frac{n}{p}\Big)a(n)+\frac{1}{p}a\Big(\frac{n}{p^{2}}\Big)\bigg)q^{n}.
\end{equation*}
Moreover, as in the integer weight case, we have 
\begin{equation*}
T_{\frac{1}{2}}(p^{2r})=T_{\frac{1}{2}}(p^{2r-2})T_{\frac{1}{2}}(p^{2})-\frac{1}{p}T_{\frac{1}{2}}(p^{2r-4}).
\end{equation*}
for all $r\geq2$. For more detail, we refer to \cite{Kob}.

\section{Multiplicative Hecke operators}

Before proving our results, we require the following simple lemma.
\begin{lemma}\label{mainlemma}
Suppose that
\begin{equation*}
f(\t)=q^{h}\prod\limits_{n=1}^{\i}(1-q^{n})^{c(n)}\in\M_{k.h}(N) \;\; \text{ and }\;\; g(\t)=q^{h'}\prod_{n=1}^{\i}(1-q^{n})^{c(n)}\in\M_{k',h'}(N).
\end{equation*}
Then, $h=h'$ and $k=k'$. In particular, we have $f(\t)=g(\t)$. 
\end{lemma}

\begin{proof}
We consider $f(\t)/g(\t)$. As $f(\t)$ and $g(\t)$ are meromorphic modular forms on $\G_{0}(N)$, $f(\t)/g(\t)=q^{h-h'}$ is also a meromorphic modular form of weight $k-k'$ on $\G_{0}(N)$. This indicates that $h-h'$ should be zero, so that $f(\t)/g(\t)$ is a constant function. Otherwise, $q^{h-h'}$ cannot be a meromorphic modular form on $\G_{0}(N)$. In particular, it follows that $k=k'$ and $f(\t)=g(\t)$.
\end{proof}

By Lemma \ref{mainlemma}, to show the theorems introduced in this section, it suffices to show that the $n$th exponent on both sides are equal for all $n$. For instance, to prove Theorem \ref{Hcomm}, it suffices to show that
\begin{equation*}
c_{\T(p^{r})\T(q^{s})}(n)=c_{\T(q^{s})\T(p^{r})}(n)
\end{equation*}
for each $n$.

\begin{proof}[proof of Theorem \ref{thmctprn}-(1) ]
In the proof of Theorem \ref{thmctprn}-(2), we show that the vanishing order at $i\i$ of $f$ is lifted by $\T(p^{r})$ to $h\sigma(p^{r})$ when $p$ does not divide $N$, and remains unchanged when $p$ divides $N$. Thus, it suffices to show that $f|\T(p^{r})$ transforms like a modular form on $\G_{0}(N)$ for some unitary multiplier system of $\G_{0}(N)$. For simplicity, we assume that $p\nmid N$. For each $\g\in\G_{0}(N)$ and $\begin{psmallmatrix}a&b\\0&d\end{psmallmatrix}$ appearing in the definition of $\T(p)$, we can see that there exist $\g'\in\G_{0}(N)$, $a',b',d'\geq0$ such that $a'd'=p$, $0\leq b'<d'$ and
\begin{equation*}
\begin{pmatrix}a&b\\0&d\end{pmatrix}\g=\g'\begin{pmatrix}a'&b'\\0&d'\end{pmatrix}.
\end{equation*}

Then for all $\g\in\G_{0}(N)$ and $f\in\M_{k,h}(N)$ with a unitary multiplier system $\chi$, we deduce that
\begin{align*}
f|\T(p)|\g&=\ve p^{k(p-1)/2}\prod_{\substack{ad=p\\0\leq b<d}}f|_{k}\begin{pmatrix}a&b\\0&d\end{pmatrix}|_{k}\g=\ve p^{k(p-1)/2}\prod_{\substack{a'd'=p\\0\leq b'<d'}}f|_{k}\g'|_{k}\begin{pmatrix}a'&b'\\0&d'\end{pmatrix}
\\
&=\ve p^{k(p-1)/2}\chi'(\g)\prod_{\substack{ad=p\\0\leq b<d}}f|_{k}\begin{pmatrix}a&b\\0&d\end{pmatrix}=\chi'(\g)f|\T(p)
\end{align*}
where $\chi'(\g)$ denotes a product of $\chi(\g')^{\prime}$s. This implies that $f|\T(p)$ transforms like a modular form on $\G_{0}(N)$ with the multiplier system $\chi'$. Next suppose that $p|N$. In this case, for each $\g\in\G_{0}(N)$ and $\begin{psmallmatrix}1&j\\0&p\end{psmallmatrix}$ appearing in the definition of $\T(p)$, we see that there exist $\g''\in\G_{0}(N)$, $j''\in\Z$ such that $0\leq j'' <p$ and 
\begin{equation*}
\begin{pmatrix}1&j\\0&p\end{pmatrix}\g=\g''\begin{pmatrix}1&j''\\0&p\end{pmatrix}.
\end{equation*}
The rest of the proof can be shown in a similar manner. We note that for each $n\in\N$, Theorem \ref{Heckealgebra} ensures that $f|\T(n)$ also transforms like a modular form on $\G_{0}(N)$ for some unitary multiplier system.
\end{proof}

\subsection{When $p\nmid N$}
In this subsection, $n,p$ in $\T(n),\T(p)$ are always natural numbers or primes that are coprime to level $N$.

\begin{thm}{\cite[Theorem 1]{Guer}}
Let 
\begin{equation*}
f(\t)=q^{h}\prod_{n=1}^{\i}(1-q^{n})^{c(n)}\in\M(N).
\end{equation*}
Let $p$ be a prime and $\T(p)$ be a multiplicative Hecke operator. Then, $f|\T(p)$ is given by
\begin{equation*}
f|\T(p)=q^{h(p+1)}\prod_{n=1}^{\i}(1-q^{n})^{pc(pn)+c(\frac{n}{p})+\chi_{p}(n)c(n)}
\end{equation*}
where $c(n/p)=0$ unless $p|n$.
\end{thm}

\begin{proof}
By the definition of multiplicative Hecke operator, we have
\begin{align*}
f|\T(p)&=\ve q^{ph}\prod\limits_{n=1}^{\i}(1-q^{pn})^{c(n)}\prod\limits_{j=0}^{p-1}\bigg(\zeta_{p}^{jh}q^{\frac{h}{p}}\prod\limits_{n=1}^{\i}(1-\zeta_{p}^{jn}q^{\frac{n}{p}})^{c(n)}\bigg)
\\
&=q^{h(p+1)}\prod_{n=1}^{\i}(1-q^{pn})^{c(n)}\prod_{j=0}^{p-1}\prod_{n=1}^{\i}(1-\zeta_{p}^{jn}q^{\frac{n}{p}})^{c(n)}
\\
&=q^{h(p+1)}\prod_{n=1}^{\i}(1-q^{pn})^{c(n)}\prod_{\substack{n=1\\p|n}}^{\i}(1-q^{\frac{n}{p}})^{pc(n)}\prod_{\substack{n=1\\(n,p)=1}}^{\i}\prod_{j=0}^{p-1}(1-\zeta_{p}^{jn}q^{\frac{n}{p}})^{c(n)}.
\end{align*}
Note that the second equality follows from the choice of $\ve$. Since $\prod_{j=0}^{p-1}(1-\zeta_{p}^{j}X)=1-X^{p}$, we have
\begin{align*}
f|\T(p)&=q^{h(p+1)}\prod_{n=1}^{\i}(1-q^{pn})^{c(n)}\prod_{n=1}^{\i}(1-q^{n})^{pc(pn)}\prod_{\substack{n=1\\(n,p)=1}}^{\i}(1-q^{n})^{c(n)}
\\
&=q^{h(p+1)}\prod_{n=1}^{\i}(1-q^{n})^{c(\frac{n}{p})+pc(pn)+\chi_{p}(n)c(n)}.
\end{align*}
\end{proof}

We now prove Theorem \ref{thmctprn}-(2).

\begin{proof}[Proof of Theorem \ref{thmctprn}-(2)]
One can compute the following even though it is tedious.
\begin{align*}
f|\T(p^{r})&=\ve f(p^{r}\t)\prod_{0\leq j<p}f\bigg(\frac{p^{r-1}\t+j}{p}\bigg)\prod_{0\leq j<p^{2}}f\bigg(\frac{p^{r-2}\t+j}{p^{2}}\bigg)\cdots\prod_{0\leq j<p^{r-1}}f\bigg(\frac{p\t+j}{p^{r-1}}\bigg)\prod_{0\leq j<p^{r}}f\bigg(\frac{\t+j}{p^{r}}\bigg)
\\
&=\ve q^{p^{r}h}\prod_{n=1}^{\i}(1-q^{p^{r}n})^{c(n)}\prod_{j=0}^{p-1}\bigg(\zeta_{p}^{jh}q^{\frac{h}{p}p^{r-1}}\prod_{n=1}^{\i}\big(1-\zeta_{p}^{jn}q^{\frac{n}{p}p^{r-1}}\big)^{c(n)}\bigg)
\\
&\times\prod_{j=0}^{p^{2}-1}\bigg(\zeta_{p^{2}}^{jh}q^{\frac{h}{p^{2}}p^{r-2}}\prod_{n=1}^{\i}\big(1-\zeta_{p^{2}}^{jn}q^{\frac{n}{p^{2}}p^{r-2}}\big)^{c(n)}\bigg)\cdots
\\
&\times\prod_{j=0}^{p^{r-1}-1}\bigg(\zeta_{p^{r-1}}^{jh}q^{\frac{h}{p^{r-1}}p}\prod_{n=1}^{\i}\big(1-\zeta_{p^{r-1}}^{jn}q^{\frac{n}{p^{r-1}}p}\big)^{c(n)}\bigg)\prod_{j=0}^{p^{r}-1}\bigg(\zeta_{p^{r}}^{jh}q^{\frac{h}{p^{r}}}\prod_{n=1}^{\i}\big(1-\zeta_{p^{r}}^{jn}q^{\frac{n}{p^{r}}}\big)^{c(n)}\bigg)
\\
&=q^{h(\frac{p^{r+1}-1}{p-1})}\prod_{n=1}^{\i}\big(1-q^{p^{r}n}\big)^{c(n)}\prod_{n=1}^{\i}\prod_{j=0}^{p-1}\big(1-\zeta_{p}^{jn}q^{\frac{n}{p}p^{r-1}}\big)^{c(n)}\prod_{n=1}^{\i}\prod_{j=0}^{p^{2}-1}\big(1-\zeta_{p^{2}}^{jn}q^{\frac{n}{p^{2}}p^{r-2}}\big)^{c(n)}
\\
&\cdots\times\prod_{n=1}^{\i}\prod_{j=0}^{p^{r-1}-1}\big(1-\zeta_{p^{r-1}}^{jn}q^{\frac{n}{p^{r-1}}p}\big)^{c(n)}\prod_{n=1}^{\i}\prod_{j=0}^{p^{r}-1}\big(1-\zeta_{p^{r}}^{jn}q^{\frac{n}{p^{r}}}\big)^{c(n)}
\\
&=q^{h(\frac{p^{r+1}-1}{p-1})}\prod_{n=1}^{\i}\big(1-q^{p^{r}n}\big)^{c(n)}\prod_{\substack{n=1\\ p|n}}^{\i}\big(1-q^{\frac{n}{p}p^{r-1}}\big)^{pc(n)}\prod_{\substack{n=1\\(p,n)=1}}^{\i}\big(1-q^{np^{r-1}}\big)^{c(n)}
\\
&\times\prod_{\substack{n=1\\p^{2}|n}}^{\i}\big(1-q^{\frac{n}{p^{2}}p^{r-2}}\big)^{p^{2}c(n)}\prod_{\substack{n=1\\p|n,\;p^{2}\nmid n}}^{\i}\big(1-q^{\frac{n}{p}p^{r-2}}\big)^{pc(n)}\prod_{\substack{n=1\\(p,n)=1}}^{\i}\big(1-q^{np^{r-2}}\big)^{c(n)}\times\cdots
\\
&\times\prod_{\substack{n=1\\p^{r-1}|n}}^{\i}\big(1-q^{\frac{n}{p^{r-1}}p}\big)^{p^{r-1}c(n)}\prod_{\substack{n=1\\p^{r-2}|n,\;p^{r-1}\nmid n}}^{\i}\big(1-q^{\frac{n}{p^{r-2}}p}\big)^{p^{r-2}c(n)}\cdots\prod_{\substack{n=1\\p|n,\;p^{2}\nmid n}}^{\i}\big(1-q^{\frac{n}{p}p}\big)^{pc(n)}
\\
&\times\prod_{\substack{n=1\\(p,n)=1}}^{\i}\big(1-q^{np}\big)^{c(n)}\prod_{\substack{n=1\\p^{r}|n}}^{\i}\big(1-q^{\frac{n}{p^{r}}}\big)^{p^{r}c(n)}\prod_{\substack{n=1\\p^{r-1}|n,\;p^{r}\nmid n}}^{\i}\big(1-q^{\frac{n}{p^{r-1}}}\big)^{p^{r-1}c(n)}\times\cdots
\\
&\times\prod_{\substack{n=1\\p|n,\;p^{2}\nmid n}}^{\i}\big(1-q^{\frac{n}{p}}\big)^{pc(n)}\prod_{\substack{n=1\\(p,n)=1}}^{\i}\big(1-q^{n}\big)^{c(n)}.
\end{align*}
Finally, we have 
\begin{align*}
&=q^{h(\frac{p^{r+1}-1}{p-1})}\prod_{n=1}^{\i}\big(1-q^{p^{r}n}\big)^{c(n)}\prod_{n=1}^{\i}\big(1-q^{np^{r-1}}\big)^{pc(pn)}\prod_{n=1}^{\i}\big(1-q^{np^{r-1}}\big)^{\chi_{p}(n)c(n)}
\\
&\times\prod_{n=1}^{\i}\big(1-q^{np^{r-2}}\big)^{p^{2}c(p^{2}n)}\prod_{n=1}^{\i}\big(1-q^{np^{r-2}}\big)^{\chi_{p}(n)pc(pn)}\prod_{n=1}^{\i}\big(1-q^{np^{r-2}}\big)^{\chi_{p}(n)c(n)}\times\cdots
\\
&\times\prod_{n=1}^{\i}\big(1-q^{np}\big)^{p^{r-1}c(p^{r-1}n)}\prod_{n=1}^{\i}\big(1-q^{np}\big)^{\chi_{p}(n)p^{r-2}c(p^{r-2}n)}\cdots\prod_{n=1}^{\i}\big(1-q^{np}\big)^{\chi_{p}(n)pc(pn)}
\\
&\times\prod_{n=1}^{\i}\big(1-q^{np}\big)^{\chi_{p}(n)c(n)}\prod_{n=1}^{\i}\big(1-q^{n}\big)^{p^{r}c(p^{r}n)}\prod_{n=1}^{\i}\big(1-q^{n}\big)^{\chi_{p}(n)p^{r-1}c(p^{r-1}n)}\times\cdots
\\
&\times\prod_{n=1}^{\i}\big(1-q^{n}\big)^{\chi_{p}(n)pc(pn)}\prod_{n=1}^{\i}\big(1-q^{n}\big)^{\chi_{p}(n)c(n)}.
\end{align*}
Hence, we have
\begin{equation*}
f|\T(p^{r})=q^{h(\frac{p^{r+1}-1}{p-1})}\prod_{n=1}^{\i}(1-q^{n})^{c_{p^{r}}(n)}
\end{equation*}
where
\begin{align*}
c_{p^{r}}(n)&=\c\bigg(1,\frac{n}{p^{r}}\bigg)+p\c\bigg(p,\frac{n}{p^{r-1}}\bigg)+\chi_{p}\bigg(\frac{n}{p^{r-1}}\bigg)\c\bigg(1,\frac{n}{p^{r-1}}\bigg)+p^{2}\c\bigg(p^{2},\frac{n}{p^{r-2}}\bigg)
\\
&+\chi_{p}\bigg(\frac{n}{p^{r-2}}\bigg)p\c\bigg(p,\frac{n}{p^{r-2}}\bigg)+\chi_{p}\bigg(\frac{n}{p^{r-2}}\bigg)\c\bigg(1,\frac{n}{p^{r-2}}\bigg)+\cdots
\\
&+p^{r-1}\c\bigg(p^{r-1},\frac{n}{p}\bigg)+\chi_{p}\bigg(\frac{n}{p}\bigg)p^{r-2}\c\bigg(p^{r-2},\frac{n}{p}\bigg)+\cdots+\chi_{p}\bigg(\frac{n}{p}\bigg)p\c\bigg(p,\frac{n}{p}\bigg)
\\
&+\chi_{p}\bigg(\frac{n}{p}\bigg)\c\bigg(1,\frac{n}{p}\bigg)+p^{r}\c(p^{r},n)+\chi_{p}(n)p^{r-1}\c(p^{r-1},n)+\cdots+\chi_{p}(n)p\c(p,n)+\chi_{p}(n)\c(1,n)
\\
&=\sum\limits_{i=0}^{r}p^{i}\c\bigg(p^{i},\frac{n}{p^{r-i}}\bigg)+\sum\limits_{k=0}^{r-1}\sum\limits_{i=0}^{k}\chi_{p}\bigg(\frac{n}{p^{r-k-1}}\bigg)p^{i}\c\bigg(p^{i},\frac{n}{p^{r-k-1}}\bigg).
\end{align*}
\end{proof}

\begin{proof}[Proof of Corollary \ref{congruence}]
We write $n=p^{m}e$ for some $m\geq0$ and $e\in\N$ such that $p\nmid e$. Then, we have
\begin{equation*}
c_{p^{r}}(n)\equiv \c\bigg(1,\frac{n}{p^{r}}\bigg)+\sum\limits_{k=0}^{r-1}\chi_{p}\bigg(\frac{n}{p^{r-k-1}}\bigg)\c\bigg(1,\frac{n}{p^{r-k-1}}\bigg)\pmod{p}
\end{equation*}
from \eqref{ctpr}. If $0\leq m<r$, then we have
\begin{equation*}
c_{p^{r}}(n)\equiv \c\bigg(1,\frac{n}{p^{m}}\bigg)\pmod{p}.
\end{equation*} 
If $m\geq r$, then we have
\begin{equation*}
c_{p^{r}}(n)\equiv \c\bigg(1,\frac{n}{p^{r}}\bigg)\pmod{p}.
\end{equation*}
\end{proof}

We now focus on Theorem \ref{Heckealgebra}. First, we prove Theorem \ref{Heckealgebra} when $m=p^{r}$ and $n=p$.

\begin{thm}\label{relation}
Let $p$ be a prime. Then the following relation
\begin{equation}\label{prp}
f|\T(p^{r})\T(p)=f|\T(p^{r+1})\cdot (f|\T(p^{r-1}))^{p}
\end{equation}
holds for all positive integers $r$.
\end{thm}

\begin{proof}
From Lemma \ref{mainlemma}, it suffices to show that 
\begin{equation*}
c_{\T(p^{r})\T(p)}(n)=c_{\T(p^{r+1})}(n)+pc_{\T(p^{r-1})}(n)
\end{equation*}
for all $n\geq1$.
\\
First, we suppose that $(n,p)=1$. Then, by using \eqref{ctpr}, we have
\begin{align*}
c_{p^{r+1}}(n)&=\sum\limits_{i=0}^{r+1}p^{i}\c\bigg(p^{i},\frac{n}{p^{r-i+1}}\bigg)+\sum\limits_{k=0}^{r}\sum\limits_{i=0}^{k}\chi_{p}\bigg(\frac{n}{p^{r-k}}\bigg)p^{i}\c\bigg(p^{i},\frac{n}{p^{r-k}}\bigg)
\\
&=p^{r+1}\c(p^{r+1},n)+\sum\limits_{i=0}^{r}p^{i}\c(p^{i},n)=\sum\limits_{i=0}^{r+1}p^{i}\c(p^{i},n).
\end{align*}
Similarly, we have 
\begin{align*}
c_{p^{r-1}}(n)=\sum\limits_{i=0}^{r-1}p^{i}\c(p^{i},n).
\end{align*}
Therefore, the $n$th exponent in the infinite product expansion of the right-hand side of \eqref{prp} is given by
\begin{equation*}
c_{p^{r+1}}(n)+pc_{p^{r-1}}(n)=\sum\limits_{i=0}^{r+1}p^{i}\c(p^{i},n)+\sum\limits_{i=0}^{r-1}p^{i+1}\c(p^{i},n).
\end{equation*}
Using the same formula, we obtain
\begin{align*}
c_{\T(p^{r})\T(p)}(n)&=c_{p^{r}}\bigg(\frac{n}{p}\bigg)+pc_{p^{r}}(pn)+\chi_{p}(n)c_{p^{r}}(n)=pc_{p^{r}}(pn)+c_{p^{r}}(n)
\\
&=p\bigg(\sum\limits_{i=0}^{r}p^{i}\c\bigg(p^{i},\frac{n}{p^{r-i-1}}\bigg)+\sum\limits_{k=0}^{r-1}\sum\limits_{i=0}^{k}\chi_{p}\bigg(\frac{n}{p^{r-k-2}}\bigg)p^{i}\c\bigg(p^{i},\frac{n}{p^{r-k-2}}\bigg)\bigg)
\\
&+\sum\limits_{i=0}^{r}p^{i}\c\bigg(p^{i},\frac{n}{p^{r-i}}\bigg)+\sum\limits_{k=0}^{r-1}\sum\limits_{i=0}^{k}\chi_{p}\bigg(\frac{n}{p^{r-k-1}}\bigg)p^{i}\c\bigg(p^{i},\frac{n}{p^{r-k-1}}\bigg).
\end{align*}
As we assume that $(n,p)=1$, the terms of the first summations are zero unless $i=r,r-1$. The other summations follow a similar rule. Thus, we have 
\begin{align*}
c_{\T(p^{r})\T(p)}(n)&=p^{r+1}\c(p^{r},pn)+p^{r}\c(p^{r-1},n)+\sum\limits_{i=0}^{r-1}\chi_{p}(pn)p^{i+1}\c(p^{i},pn)
\\
&+\sum\limits_{i=0}^{r-2}\chi_{p}(n)p^{i+1}\c(p^{i},n)+p^{r}\c(p^{r},n)+\sum\limits_{i=0}^{r-1}p^{i}\c(p^{i},n)
\\
&=\sum\limits_{i=0}^{r+1}p^{i}\c(p^{i},n)+\sum\limits_{i=0}^{r-1}p^{i+1}\c(p^{i},n).
\end{align*}
Next, suppose that $p^{m}|n$ and $p^{m+1}\nmid n$ for fixed $m\geq1$. Then, we have
\begin{align*}
c_{p^{r+1}}(n)&=\sum\limits_{i=0}^{r+1}p^{i}\c\bigg(p^{i},\frac{n}{p^{r-i+1}}\bigg)+\sum\limits_{k=0}^{r}\sum\limits_{i=0}^{k}\chi_{p}\bigg(\frac{n}{p^{r-k}}\bigg)p^{i}\c\bigg(p^{i},\frac{n}{p^{r-k}}\bigg)
\\
&=\sum\limits_{i=r-m+1}^{r+1}p^{i}\c\bigg(p^{i},\frac{n}{p^{r-i+1}}\bigg)+\sum\limits_{k=r-m}^{r}\sum\limits_{i=0}^{k}\chi_{p}\bigg(\frac{n}{p^{r-k}}\bigg)p^{i}\c\bigg(p^{i},\frac{n}{p^{r-k}}\bigg).
\end{align*}
From the definitions of $\chi_{p}$ and $\c(\empty\;,\empty\;)$, the second summation is zero, unless $k=r-m$. Hence, we have
\begin{equation*}
c_{p^{r+1}}(n)=\sum\limits_{i=r-m+1}^{r+1}p^{i}\c\bigg(p^{i},\frac{n}{p^{r-i+1}}\bigg)+\sum\limits_{i=0}^{r-m}p^{i}\c\bigg(p^{i},\frac{n}{p^{m}}\bigg).
\end{equation*}
Similarly, we have
\begin{align*}
c_{p^{r-1}}(n)=\sum\limits_{i=r-m-1}^{r-1}p^{i}\c\bigg(p^{i},\frac{n}{p^{r-i-1}}\bigg)+\sum\limits_{i=0}^{r-m-2}p^{i}\c\bigg(p^{i},\frac{n}{p^{m}}\bigg).
\end{align*}
On the other hand, because $p^{m}|n$ and $p^{m+1}\nmid n$, we have
\begin{align*}
c_{\T(p^{r})\T(p)}(n)&=c_{p^{r}}\bigg(\frac{n}{p}\bigg)+pc_{p^{r}}(pn)+\chi_{p}(n)c_{p^{r}}(n)=c_{p^{r}}\bigg(\frac{n}{p}\bigg)+pc_{p^{r}}(pn)
\\
&=\sum\limits_{i=0}^{r}p^{i}\c\bigg(p^{i},\frac{n}{p^{r-i+1}}\bigg)+\sum\limits_{k=0}^{r-1}\sum\limits_{i=0}^{k}\chi_{p}\bigg(\frac{n}{p^{r-k}}\bigg)p^{i}\c\bigg(p^{i},\frac{n}{p^{r-k}}\bigg)
\\
&+\sum\limits_{i=0}^{r}p^{i+1}\c\bigg(p^{i},\frac{n}{p^{r-i-1}}\bigg)+\sum\limits_{k=0}^{r-1}\sum\limits_{i=0}^{k}\chi_{p}\bigg(\frac{n}{p^{r-k-2}}\bigg)p^{i+1}\c\bigg(p^{i},\frac{n}{p^{r-k-2}}\bigg)
\\
&=\sum\limits_{i=r-m+1}^{r}p^{i}\c\bigg(p^{i},\frac{n}{p^{r-i+1}}\bigg)+\sum\limits_{k=r-m}^{r-1}\sum\limits_{i=0}^{k}\chi_{p}\bigg(\frac{n}{p^{r-k}}\bigg)p^{i}\c\bigg(p^{i},\frac{n}{p^{r-k}}\bigg)
\\
&+\sum\limits_{i=r-m-1}^{r}p^{i+1}\c\bigg(p^{i},\frac{n}{p^{r-i-1}}\bigg)+\sum\limits_{k=r-m-2}^{r-1}\sum\limits_{i=0}^{k}\chi_{p}\bigg(\frac{n}{p^{r-k-2}}\bigg)p^{i+1}\c\bigg(p^{i},\frac{n}{p^{r-k-2}}\bigg)
\\
&=\sum\limits_{i=r-m+1}^{r}p^{i}\c\bigg(p^{i},\frac{n}{p^{r-i+1}}\bigg)+\sum\limits_{i=0}^{r-m}p^{i}\c\bigg(p^{i},\frac{n}{p^{m}}\bigg)
\\
&+\sum\limits_{i=r-m-1}^{r}p^{i+1}\c\bigg(p^{i},\frac{n}{p^{r-i-1}}\bigg)+\sum\limits_{i=0}^{r-m-2}p^{i+1}\c\bigg(p^{i},\frac{n}{p^{m}}\bigg).
\end{align*}
Hence, 
\begin{align*}
&c_{p^{r+1}}(n)+pc_{p^{r-1}}(n)
\\
&=\sum\limits_{i=r-m+1}^{r}p^{i}\c\bigg(p^{i},\frac{n}{p^{r-i+1}}\bigg)+p^{r+1}\c(p^{r+1},n)+\sum\limits_{i=0}^{r-m}p^{i}\c\bigg(p^{i},\frac{n}{p^{m}}\bigg)
\\
&+\sum\limits_{i=r-m-1}^{r-1}p^{i+1}\c\bigg(p^{i},\frac{n}{p^{r-i-1}}\bigg)+\sum\limits_{i=0}^{r-m-2}p^{i+1}\c\bigg(p^{i},\frac{n}{p^{m}}\bigg)
\\
&=\sum\limits_{i=r-m+1}^{r}p^{i}\c\bigg(p^{i},\frac{n}{p^{r-i+1}}\bigg)+\sum\limits_{i=0}^{r-m}p^{i}\c\bigg(p^{i},\frac{n}{p^{m}}\bigg)
\\
&+\sum\limits_{i=r-m-1}^{r}p^{i+1}\c\bigg(p^{i},\frac{n}{p^{r-i-1}}\bigg)+\sum\limits_{i=0}^{r-m-2}p^{i+1}\c\bigg(p^{i},\frac{n}{p^{m}}\bigg)
\\
&=c_{\T(p^{r})\T(p)}(n).
\end{align*}
\end{proof}

\begin{thm}
Let $p$ be a prime and $r,s$ be positive integers. Then, we have
\begin{equation}\label{relation2}
f|\T(p^{r})\T(p^{s})=\prod_{d|(p^{r},p^{s})}\bigg(f|\T\bigg(\frac{p^{r+s}}{d^{2}}\bigg)\bigg)^{d}.
\end{equation}
\end{thm}

\begin{proof}
We use induction on $s$. The case $s=1$ follows from Theorem \ref{relation}. Next, we assume that \eqref{relation2} holds for $s=1,2,\cdots,k$. For convenience, we denote $g(\t):=f|\T(p^{r})$. Then, we have
\begin{align*}
f|\T(p^{r})\T(p^{k+1})&=g|\T(p^{k+1})=\frac{g|\T(p^{k})\T(p)}{\big(g|\T(p^{k-1})\big)^{p}}=\frac{f|\T(p^{r})\T(p^{k})\T(p)}{\big(f|\T(p^{r})\T(p^{k-1})\big)^{p}}
\\
&=\frac{\prod\limits_{d|(p^{r},p^{k})}\bigg(\big(f|\T(\frac{p^{r+k}}{d^{2}})\big)^{d}|\T(p)\bigg)}{\prod\limits_{e|(p^{r},p^{k-1})}\big(f|\T(\frac{p^{r+k-1}}{e^{2}})\big)^{ep}}=\frac{\prod\limits_{d|(p^{r},p^{k})}\big(f|\T(\frac{p^{r+k}}{d^{2}})|\T(p)\big)^{d}}{\prod\limits_{e|(p^{r},p^{k-1})}\big(f|\T(\frac{p^{r+k-1}}{e^{2}})\big)^{ep}}.
\end{align*}
When $r>k$, we have
\begin{align*}
f|\T(p^{r})\T(p^{k+1})&=\frac{\prod\limits_{i=0}^{k}\big(f|\T(p^{r+k-2i})|\T(p)\big)^{p^{i}}}{\prod\limits_{i=0}^{k-1}\big(f|\T(p^{r+k-1-2i})\big)^{p^{i+1}}}=\frac{\prod\limits_{i=0}^{k}\big(f|\T(p^{r+k-2i+1})\big)^{p^{i}}\big(f|\T(p^{r+k-2i-1})\big)^{p^{i+1}}}{\prod\limits_{i=0}^{k-1}\big(f|\T(p^{r+k-1-2i})\big)^{p^{i+1}}}
\\
&=\prod\limits_{i=0}^{k+1}\big(f|\T(p^{r+k-2i+1})\big)^{p^{i}}=\prod\limits_{d|(p^{r},p^{k+1})}f|\T\bigg(\frac{p^{r+k+1}}{d^{2}}\bigg)^{d}.
\end{align*}
When $r=k$, we have
\begin{align*}
f|\T(p^{r})\T(p^{k+1})&=\frac{\big(f|\T(p)\big)^{p^{k}}\prod\limits_{i=0}^{k-1}\big(f|\T(p^{2k-2i})|\T(p)\big)^{p^{i}}}{\prod\limits_{i=0}^{k-1}\big(f|\T(p^{2k-1-2i})\big)^{p^{i+1}}}
\\
&=\frac{\big(f|\T(p)\big)^{p^{k}}\prod\limits_{i=0}^{k-1}\big(f|\T(p^{2k-2i+1})\big)^{p^{i}}\big(f|\T(p^{2k-2i-1})\big)^{p^{i+1}}}{\prod\limits_{i=0}^{k-1}\big(f|\T(p^{2k-1-2i})\big)^{p^{i+1}}}
\\
&=\prod\limits_{i=0}^{k}\big(f|\T(p^{2k-2i+1})\big)^{p^{i}}=\prod\limits_{d|(p^{r},p^{k+1})}f|\T\bigg(\frac{p^{r+k+1}}{d^{2}}\bigg)^{d}.
\end{align*}
When $r\leq k-1$, we have 
\begin{align*}
f|\T(p^{r})\T(p^{k+1})&=\frac{\prod\limits_{i=0}^{r}\big(f|\T(p^{r+k-2i})|\T(p)\big)^{p^{i}}}{\prod\limits_{i=0}^{r}\big(f|\T(p^{r+k-1-2i})\big)^{p^{i+1}}}=\frac{\prod\limits_{i=0}^{r}\big(f|\T(p^{r+k-2i+1})\big)^{p^{i}}\big(f|\T(p^{r+k-2i-1})\big)^{p^{i+1}}}{\prod\limits_{i=0}^{r}\big(f|\T(p^{r+k-1-2i})\big)^{p^{i+1}}}
\\
&=\prod\limits_{i=0}^{r}\big(f|\T(p^{r+k-2i+1})\big)^{p^{i}}=\prod\limits_{d|(p^{r},p^{k+1})}f|\T\bigg(\frac{p^{r+k+1}}{d^{2}}\bigg)^{d}.
\end{align*}
Thus, by induction on $s$, we obtain \eqref{relation2} for all $s\geq1$.
\end{proof}

\begin{proof}[Proof of Theorem \ref{Heckealgebra} when $(m,N)=(n,N)=1$]
Let $m:=\prod_{i=1}^{u}p_{i}^{r_{i}}$ and $n:=\prod_{j=1}^{v}p_{j}^{s_{j}}$. We reorder them such that $1\leq r_{i},s_{i}$ for $1\leq i\leq l$ and $r_{i}s_{i}=0$ for $l<i$.
\begin{align*}
f|\T(m)\T(n)=f|\T(p_{1}^{r_{1}})\T(p_{1}^{s_{1}})\T(p_{2}^{r_{2}})\T(p_{2}^{s_{2}})\cdots\T(p_{l}^{r_{l}})\T(p_{l}^{s_{l}})\cdots\T(p_{w}^{r_{w}})\T(p_{w}^{s_{w}})
\end{align*}
where $w:={\rm max}\{u,v\}$. Thus, we have
\begin{align*}
&f|\T(m)\T(n)
\\
&=\prod_{d_{1}|(p_{1}^{r_{1}},p_{1}^{s_{1}})}\cdots\prod_{d_{l}|(p_{l}^{r_{l}},p_{l}^{s_{l}})}\bigg(f|\T\bigg(\frac{p_{1}^{r_{1}+s_{1}}}{d_{1}^{2}}\bigg)\cdots\T\bigg(\frac{p_{l}^{r_{l}+s_{l}}}{d_{l}^{2}}\bigg)\T(p_{l+1}^{r_{l+1}})\T(p_{l+1}^{s_{l+1}})\cdots\T(p_{w}^{r_{w}})\T(p_{w}^{s_{w}})\bigg)^{d_{1}\cdots d_{l}}
\\
&=\prod_{d_{1}|(p_{1}^{r_{1}},p_{1}^{s_{1}})}\cdots\prod_{d_{l}|(p_{l}^{r_{l}},p_{l}^{s_{l}})}\bigg(f|\T\bigg(\frac{p_{1}^{r_{1}+s_{1}}}{d_{1}^{2}}\cdots\frac{p_{l}^{r_{l}+s_{l}}}{d_{l}^{2}}p_{l+1}^{r_{l+1}}p_{l+1}^{s_{l+1}}\cdots p_{w}^{r_{w}}p_{w}^{s_{w}}\bigg)\bigg)^{d_{1}\cdots d_{l}}
\\
&=\prod_{d|(m,n)}\bigg(f|\T\bigg(\frac{mn}{d^{2}}\bigg)\bigg)^{d}
\end{align*}
\end{proof}

We now prove Theorem \ref{main3}.

\begin{proof}[Proof of Corollary \ref{main3}]
We choose any meromorphic modular form of weight $k\ne0$ on $\SL(\Z)$ and substitute it into $\eqref{mn}$. Then, comparing the weights on both sides of \eqref{mn} yields Corollary \ref{main3}.
\end{proof}

Now, we prove that multiplicative Hecke operators are commutative even in the power of prime cases.

\begin{thm}\label{Hcomm}
Let $p,q$ be distinct primes and $r,s$ be positive integers. Then, $\T(p^{r})\T(q^{s})=\T(q^{s})\T(p^{r})$.
\end{thm}

\begin{proof}
We denote $\c_{\T(p^{r})}(*,*)$ by
\begin{equation*}
\c_{\T(p^{r})}(X,Y/Z):=
\begin{cases}
c_{p^{r}}\big(\frac{XY}{Z}\big) & \text{ if } Z|Y, 
\\
0 & \text{otherwise}.
\end{cases}
\end{equation*}
Then, we have
\begin{align*}
c_{\T(p^{r})\T(q^{s})}(n)&=\sum\limits_{i=0}^{s}q^{i}\c_{\T(p^{r})}\bigg(q^{i},\frac{n}{q^{s-i}}\bigg)+\sum\limits_{k=0}^{s-1}\sum\limits_{i=0}^{k}\chi_{q}\bigg(\frac{n}{q^{s-k-1}}\bigg)q^{i}\c_{\T(p^{r})}\bigg(q^{i},\frac{n}{q^{s-k-1}}\bigg)
\\
&=\sum\limits_{i=0}^{s}q^{i}\bigg(\sum\limits_{j=0}^{r}p^{j}\c\bigg(p^{j}q^{i},\frac{n}{p^{r-j}q^{s-i}}\bigg)+\sum\limits_{l=0}^{r-1}\sum\limits_{j=0}^{l}\chi_{p}\bigg(\frac{n}{p^{r-l-1}}\bigg)p^{j}\c\bigg(p^{j}q^{i},\frac{n}{p^{r-l-1}q^{s-i}}\bigg)\bigg)
\\
&+\sum\limits_{k=0}^{s-1}\sum\limits_{i=0}^{k}\chi_{q}\bigg(\frac{n}{q^{s-k-1}}\bigg)q^{i}\sum\limits_{j=0}^{r}p^{j}\c\bigg(p^{j}q^{i},\frac{n}{p^{r-j}q^{s-k-1}}\bigg)
\\
&+\sum\limits_{k=0}^{s-1}\sum\limits_{i=0}^{k}\chi_{q}\bigg(\frac{n}{q^{s-k-1}}\bigg)q^{i}\sum\limits_{l=0}^{r-1}\sum\limits_{j=0}^{l}\chi_{p}\bigg(\frac{n}{p^{r-l-1}}\bigg)p^{j}\c\bigg(p^{j}q^{i},\frac{n}{p^{r-l-1}q^{s-k-1}}\bigg).
\end{align*}
The second equality follows from 
\begin{equation*}
\chi_{p}\bigg(\frac{n}{p^{r-l-1}q^{s-i}}\bigg)\c\bigg(p^{j}q^{i},\frac{n}{p^{r-l-1}q^{s-i}}\bigg)=\chi_{p}\bigg(\frac{n}{p^{r-l-1}}\bigg)\c\bigg(p^{j}q^{i},\frac{n}{p^{r-l-1}q^{s-i}}\bigg).
\end{equation*}
On the other hand, we have
\begin{align*}
c_{\T(q^{s})\T(p^{r})}(n)&=\sum\limits_{j=0}^{r}p^{j}\c_{\T(q^{s})}\bigg(p^{j},\frac{n}{p^{r-j}}\bigg)+\sum\limits_{l=0}^{r-1}\sum\limits_{j=0}^{l}\chi_{p}\bigg(\frac{n}{p^{r-l-1}}\bigg)p^{j}\c_{\T(q^{s})}\bigg(p^{j},\frac{n}{p^{r-l-1}}\bigg)
\\
&=\sum\limits_{j=0}^{r}p^{j}\bigg(\sum\limits_{i=0}^{s}q^{i}\c\bigg(p^{j}q^{i},\frac{n}{p^{r-j}q^{s-i}}\bigg)+\sum\limits_{k=0}^{s-1}\sum\limits_{i=0}^{k}\chi_{q}\bigg(\frac{n}{q^{s-k-1}}\bigg)q^{i}\c\bigg(p^{j}q^{i},\frac{n}{p^{r-j}q^{s-k-1}}\bigg)\bigg)
\\
&+\sum\limits_{l=0}^{r-1}\sum\limits_{j=0}^{l}\chi_{p}\bigg(\frac{n}{p^{r-l-1}}\bigg)p^{j}\sum\limits_{i=0}^{s}q^{i}\c\bigg(p^{j}q^{i},\frac{n}{p^{r-l-1}q^{s-i}}\bigg)
\\
&+\sum\limits_{l=0}^{r-1}\sum\limits_{j=0}^{l}\chi_{p}\bigg(\frac{n}{p^{r-l-1}}\bigg)p^{j}\sum\limits_{k=0}^{s-1}\sum\limits_{i=0}^{k}\chi_{q}\bigg(\frac{n}{q^{s-k-1}}\bigg)q^{i}\c\bigg(p^{j}q^{i},\frac{n}{p^{r-l-1}q^{s-k-1}}\bigg).
\end{align*}
Hence, it follows that $\T(p^{r})\T(q^{s})=\T(q^{s})\T(p^{r})$.
\end{proof}

Finally, we define $\T(n)$ for each $n\in\N$ such that $(n,N)=1$ as
\begin{equation*}
\T(n):=\prod_{p_{i}|n}\T(p_{i}^{r_{i}})
\end{equation*}
where $n=\prod p_{i}^{r_{i}}$, This is well defined by Theorem \ref{Hcomm}.
 
\subsection{When $p|N$} 

In this subsection, we investigate the multiplicative Hecke operator $\T(n)$ for which $(n,N)>1$.

\begin{proof}[Proof of Theorem \ref{thmctprn}-(2) when $p|N$]
If $p|N$, then we have
\begin{align*}
f|\T(p^{r})&=\varepsilon \prod\limits_{j=0}^{p^{r}-1}f\bigg(\frac{\t+j}{p^{r}}\bigg)=\varepsilon\prod_{j=0}^{p^{r}-1}\bigg(\zeta_{p^{r}}^{jh}q^{\frac{h}{p^{r}}}\prod_{n=1}^{\i}\big(1-\zeta_{p^{r}}^{jn}q^{\frac{n}{p^{r}}}\big)^{c(n)}\bigg)=q^{h}\prod\limits_{n=1}^{\i}\prod\limits_{j=0}^{p^{r}-1}(1-\zeta_{p^{r}}^{jn}q^{\frac{n}{p^{r}}})^{c(n)}
\\
&=q^{h}\prod_{\substack{n=1\\p^{r}|n}}^{\i}\big(1-q^{\frac{n}{p^{r}}}\big)^{p^{r}c(n)}\prod_{\substack{n=1\\p^{r-1}|n,\;p^{r}\nmid n}}^{\i}\big(1-q^{\frac{n}{p^{r-1}}}\big)^{p^{r-1}c(n)}\times\cdots\times\prod_{\substack{n=1\\(p,n)=1}}^{\i}\big(1-q^{n}\big)^{c(n)}
\\
&=q^{h}\prod_{n=1}^{\i}\big(1-q^{n}\big)^{p^{r}c(p^{r}n)}\prod_{n=1}^{\i}\big(1-q^{n}\big)^{\chi_{p}(n)p^{r-1}c(p^{r-1}n)}\times\cdots\times\prod_{n=1}^{\i}\big(1-q^{n}\big)^{\chi_{p}(n)c(n)}.
\end{align*}
Therefore, we get
\begin{equation*}
f|\T(p^{r})=q^{h}\prod\limits_{n=1}^{\i}(1-q^{n})^{c_{p^{r}}(n)}
\end{equation*}
where 
\begin{equation*}
c_{p^{r}}(n)=p^{r}\c(p^{r},n)+\sum\limits_{i=0}^{r-1}\chi_{p}(n)p^{i}\c(p^{i},n).
\end{equation*}
\end{proof}

To define $\T(n)$ for any positive integer $n$, we need to establish the commutativity of the multiplicative Hecke operators regardless of whether $n$ is coprime to the level $N$ or not. This is demonstrated in Theorems \ref{prpsN}--\ref{plsN} below, which show that these operators commute whenever at least one of them is not coprime to level $N$. 

\begin{thm}\label{prpsN}
Let $p$ be a prime such that $p|N$ and $r,s$ be positive integers. Then, we have
\begin{equation*}
\T(p^{r})\T(p^{s})=\T(p^{s})\T(p^{r})=\T(p^{r+s}).
\end{equation*}
\end{thm}

\begin{proof}
It suffices to show that $c_{\T(p^{r})\T(p^{s})}(n)=c_{\T(p^{r+s})}(n)$ for all $n$. We have
\begin{align*}
c_{\T(p^{r})\T(p^{s})}(n)&=p^{s}\c_{\T(p^{r})}(p^{s},n)+\sum\limits_{i=0}^{s-1}\chi_{p}(n)p^{i}\c_{\T(p^{r})}(p^{i},n)
\\
&=p^{s}\bigg(p^{r}\c(p^{r},p^{s}n)+\sum\limits_{j=0}^{r-1}\chi_{p}(p^{s}n)p^{j}\c(p^{j},p^{s}n)\bigg)
\\
&+\sum\limits_{i=0}^{s-1}\chi_{p}(n)p^{i}\bigg(p^{r}\c(p^{r},p^{i}n)+\sum\limits_{j=0}^{r-1}\chi_{p}(p^{i}n)p^{j}\c(p^{j},p^{i}n)\bigg)
\\
&=p^{r+s}\c(p^{r},p^{s}n)+\sum\limits_{i=0}^{s-1}\chi_{p}(n)p^{i+r}\c(p^{i+r},n)+\sum\limits_{j=0}^{r-1}\chi_{p}(n)p^{j}\c(p^{j},n)
\\
&=p^{r+s}\c(p^{r},p^{s}n)+\sum\limits_{i=0}^{r+s-1}\chi_{p}(n)p^{i}\c(p^{i},n)=c_{\T(p^{r+s})}(n).
\end{align*}
\end{proof}

\begin{thm}\label{plN}
Let $p$ and $l$ be distinct primes dividing $N$. Then, $\T(p)\T(l)=\T(l)\T(p)$.
\end{thm}

\begin{proof}
It suffices to show that switching $p$ and $l$ does not change $c_{\T(p)\T(l)}(n)$ for all $n$. From \eqref{ctpr}, we have 
\begin{align*}
c_{\T(p)\T(l)}(n)&=lc_{\T(p)}(ln)+\chi_{l}(n)c_{\T(p)}(n)
\\
&=l\Big(p\c(p,ln)+\chi_{p}(ln)\c(1,ln)\Big)+\chi_{l}(n)\Big(p\c(p,n)+\chi_{p}(n)\c(1,n)\Big).
\end{align*}
Note that $p$ and $l$ are distinct primes. Thus, the right-hand side of the above equation does not change when $p$ and $l$ are swapped.
\end{proof}

\begin{cor}\label{prlsN}
Let $p$ and $l$ be distinct primes dividing $N$. Let $r,s$ be positive integers. Then, 
\begin{equation*}
\T(p^{r})\T(l^{s})=\T(l^{s})\T(p^{r})
\end{equation*}
\end{cor}

\begin{proof}
This follows from Theorem \ref{prpsN} and \ref{plN}.
\end{proof}

We now consider the case where one of the primes divides level $N$ and the other does not divide $N$.

\begin{thm}\label{plsN}
Let $p$ and $l$ be distinct primes, such that $p|N$ and $l\nmid N$. Let $s$ be a positive integer. Then, we have
\begin{equation*}
\T(p)\T(l^{s})=\T(l^{s})\T(p).
\end{equation*}
In particular, $\T(p^{r})\T(l^{s})=\T(l^{s})\T(p^{r})$ for all $r\geq1$.
\end{thm}

\begin{proof}
We get
\begin{align*}
c_{\T(p)\T(l^{s})}(n)&=\sum\limits_{i=0}^{s}l^{i}\c_{\T(p)}\bigg(l^{i},\frac{n}{l^{s-i}}\bigg)+\sum\limits_{k=0}^{s-1}\sum\limits_{i=0}^{k}\chi_{l}\bigg(\frac{n}{l^{s-k-1}}\bigg)l^{i}\c_{\T(p)}\bigg(l^{i},\frac{n}{l^{s-k-1}}\bigg)
\\
&=\sum\limits_{i=0}^{s}l^{i}\bigg(p\c\bigg(pl^{i},\frac{n}{l^{s-i}}\bigg)+\chi_{p}\bigg(\frac{n}{l^{s-i}}\bigg)\c\bigg(l^{i},\frac{n}{l^{s-i}}\bigg)\bigg)
\\
&+\sum\limits_{k=0}^{s-1}\sum\limits_{i=0}^{k}\chi_{l}\bigg(\frac{n}{l^{s-k-1}}\bigg)l^{i}\bigg(p\c\bigg(pl^{i},\frac{n}{l^{s-k-1}}\bigg)+\chi_{p}\bigg(\frac{n}{l^{s-k-1}}\bigg)\c\bigg(l^{i},\frac{n}{l^{s-k-1}}\bigg)\bigg)
\\
&=\sum\limits_{i=0}^{s}l^{i}p\c\bigg(pl^{i},\frac{n}{l^{s-i}}\bigg)+l^{i}\chi_{p}(n)\c\bigg(l^{i},\frac{n}{l^{s-i}}\bigg)
\\
&+\sum\limits_{k=0}^{s-1}\sum\limits_{i=0}^{k}\chi_{l}\bigg(\frac{n}{l^{s-k-1}}\bigg)l^{i}p\c\bigg(pl^{i},\frac{n}{l^{s-k-1}}\bigg)+\chi_{l}\bigg(\frac{n}{l^{s-k-1}}\bigg)l^{i}\chi_{p}(n)\c\bigg(l^{i},\frac{n}{l^{s-k-1}}\bigg).
\end{align*}
On the other hand, we have
\begin{align*}
c_{\T(l^{s})\T(p)}(n)&=p\c_{\T(l^{s})}(p,n)+\chi_{p}(n)\c_{\T(l^{s})}(1,n)
\\
&=\sum\limits_{i=0}^{s}pl^{i}\c\bigg(l^{i}p,\frac{n}{l^{s-i}}\bigg)+\sum\limits_{k=0}^{s-1}\sum\limits_{i=0}^{k}\chi_{l}\bigg(\frac{n}{l^{s-k-1}}\bigg)pl^{i}\c\bigg(l^{i}p,\frac{n}{l^{s-k-1}}\bigg)
\\
&+\sum\limits_{i=0}^{s}l^{i}\chi_{p}(n)\c\bigg(l^{i},\frac{n}{l^{s-i}}\bigg)+\sum\limits_{k=0}^{s-1}\sum\limits_{i=0}^{k}\chi_{p}(n)\chi_{l}\bigg(\frac{n}{l^{s-k-1}}\bigg)l^{i}\c\bigg(l^{i},\frac{n}{l^{s-k-1}}\bigg).
\end{align*}
This proves the first desired result. Next, from Theorem \ref{prpsN}, we obtain $\T(p^{r})\T(l^{s})=\T(l^{s})\T(p^{r})$ for all $r\geq1$.
\end{proof}

Finally, we define the multiplicative Hecke operator $\T(n)$ for all positive integers $n$ by combining Theorem \ref{Hcomm}, Corollary \ref{prlsN}, and Theorem \ref{plsN} as follows:
\begin{equation*}
\T(n):=\prod\limits_{p_{i}|n}\T(p_{i}^{r_{i}})
\end{equation*}
where $n=\prod_{p_{i}|n}p_{i}^{r_{i}}$. We note that this is well defined.

\begin{proof}[Proof of Theorem \ref{Heckealgebra} for arbitrary $m$ and $n$]
\ \newline
Let $m=\prod_{p_{i}|N}p_{i}^{r_{i}}\prod_{q_{i}\nmid N}q_{i}^{s_{i}}$ and $n=\prod_{p_{i}|N}p_{i}^{r_{i}'}\prod_{q_{i}\nmid N}q_{i}^{s_{i}'}$ with $r_{i},s_{i},r_{i}'$ and $s_{i}'\geq0$.
We denote $\mathfrak{p}:=\prod_{p_{i}|N}p_{i}^{r_{i}}$ and $\mathfrak{p}':=\prod_{p_{i}|N}p_{i}^{r_{i}'}$.
Then, we have
\begin{align*}
f|\T(m)\T(n)&=f|\prod\limits_{p_{i}|N}\T(p_{i}^{r_{i}+r_{i}'})\prod\limits_{q_{i}\nmid N}\T(q_{i}^{s_{i}})\T(q_{i}^{s_{i}'})
\\
&=\prod_{d|(\frac{m}{\mathfrak{p}},\frac{n}{\mathfrak{p}'})}\bigg(f|\prod\limits_{p_{i}|N}\T(p_{i}^{r_{i}+r_{i}'})|\T\bigg(\frac{mn/\mathfrak{pp'}}{d^{2}}\bigg)\bigg)^{d}
\\
&=\prod_{d|(m,n)}\bigg(\chi_{N}(d)f|\T\bigg(\frac{mn}{d^{2}}\bigg)\bigg)^{d}.
\end{align*}
The second equality follows from the proof of the same theorem under the condition that $(m,N)=(n,N)=1$ in the previous subsection.
\end{proof}

\section{Multiplicative Hecke eigenforms}

In this section, we prove Theorems \ref{eigenform} and \ref{equivalent}. To prove these, we require the formula related to the exponents in the infinite product expansion of the multiplicative Hecke eigenform. The following lemma provides a more concrete description of the multiplicative Hecke eigenform.

\begin{lemma}\label{coef}
Let $N$ be a positive integer and $f(\t)\in\M_{k,h}(N)$. Then,
\begin{equation*}
f(\t)=q^{h}\prod\limits_{n=1}^{\i}(1-q^{n})^{c(n)}
\end{equation*}
is a multiplicative Hecke eigenform if and only if $c(n)=c(pn)$ for all positive integers $n$ and primes $p$ not dividing the level $N$.
\end{lemma}

\begin{proof}
First, suppose that $f(\t)$ is a multiplicative Hecke eigenform on $\G_{0}(N)$. Then, by Theorem \ref{ctpr}, we have
\begin{equation}\label{coe}
\sigma(p)c(n)=pc(pn)+\chi_{p}(n)c(n)+c(n/p)
\end{equation}
for all positive integers $n$ and prime $p$ for which $(p,N)=1$. Substituting $n=p^{r}e$ such that $(e,p)=1$ into \eqref{coe} implies $c(p^{r+1}e)=c(e)$ for all nonnegative integers $r$ and prime $p$ for which $(p,N)=1$. Next, in the reverse direction, it suffices to show that $f^{\sigma(p)}=f|\T(p)$ for all primes $p$ not dividing $N$. We assume that $c(n)=c(pn)$ holds for all positive integers $n$ and primes $p$ for which $(p,N)=1$. Then, we obtain $\chi_{p}(n)c(n)=c(n)$(resp. $c(n/p)=c(n)$) when $(n,p)=1$(resp. $(n,p)=p$). This indicates that the right-hand side of \eqref{coe} is equal to $\sigma(p)c(n)$. In other words, $f^{\sigma(p)}=f|\T(p)$ for all primes $p$ such that $(p,N)=1$.
\end{proof}

Lemma \ref{coef} implies that the multiplicative Hecke eigenform is determined by the exponents $c(p_{i}^{r_{i}})$ in its infinite product expansion, where $p_{i}$ is the prime dividing $N$ and $r_{i}\geq0$. Now, we prove Theorem \ref{eigenform} and Theorem \ref{equivalent}.

\begin{proof}[Proof of Theorem \ref{eigenform}]
Let $f(\t)$ be an eta quotient of level $N$ of the form
\begin{equation*}
f(\t)=\prod\limits_{\delta|N}\eta(\delta\t)^{r_{\delta}}=q^{h}\prod\limits_{n=1}^{\i}(1-q^{n})^{c(n)}.
\end{equation*}
For each positive integer $n$, we put $d:=(n,N)$. Then, we have 
\begin{equation*}
c(n)=\sum\limits_{\delta|d}r_{\delta}.
\end{equation*}
Let $p$ be a prime such that $(p,N)=1$. Then, we have $c(pn)=c(n)$ for all positive integers $n$ because $(pn,N)$ is also equal to $d$. Finally, the desired result follows from Lemma \ref{coef}.
\end{proof}

\begin{proof}[Proof of Theorem \ref{equivalent}]
By \cite[Corollary 8]{RW} and Theorem \ref{eigenform}, it suffices to show that if $f(\t)$ is a multiplicative Hecke eigenform, then it has no zeros or poles on $\H$. Let $f(\t)\in\M(N)$ be a multiplicative Hecke eigenform. Suppose that $f(\t)$ has zeros or poles in $\mathcal{F}$ where $\mathcal{F}$ is the fundamental domain for $\G_{0}(N)$. Let $z_{1},\cdots,z_{n}$ be the zeros or poles of $f(\t)$ in $\mathcal{F}$. Then, the set of zeros or poles of $f|\T(p)$ in $\mathcal{F}$ should be equal to the set $\{z_{1},\cdots z_{n}\}$ because $f|\T(p)=f^{\sigma(p)}$ for all $p$ not dividing $N$. For a fixed $z_{i}\in\H$ and prime $p$, we claim that there are infinitely many primes $l\ne p$ such that $\g (pz_{i})\ne lz_{i}$ for all $\g=\begin{psmallmatrix}a&b\\c&d\end{psmallmatrix}\in\G_{0}(N)$. In other words, for a fixed point $z_{i}\in\H$, the set $\{pz_{i}:p \text{ is a prime}\}/\sim_{\G_{0}(N)}$ has infinitely many elements, where $\sim_{\G_{0}(N)}$ indicates the $\G_{0}(N)$-equivalence. To prove this claim, we consider the equation $\g pz_{i}=lz_{i}$, where $\g\in\G_{0}(N)$. We have
\begin{equation*}
\frac{apz_{i}+b}{cpz_{i}+d}=lz_{i}.
\end{equation*}
From the above equation, we obtain
\begin{equation*}
ap=l(2cpx+d) ,\;\;\; \text{and} \;\;\; b=l\Big(cp(x^{2}-y^{2})+dx-x(2cpx+d)\Big)
\end{equation*}
where $z_{i}=x+yi$($x,y\in\R$). The first equation does not hold unless $x\in\Q$ (specifically, $x\in\frac{1}{2cpl}\Z$). Therefore, we may assume that $x\in\Q$. We denote $x=\alpha/\beta$ with $\alpha,\beta\ne0\in\Z$ and $(\alpha,\beta)=1$. Then, we have $l|a$ for all but finitely many primes $l$ because there are only finitely many primes dividing $\beta$. From the same argument, we have $l|b$. This is a contradiction because $ad-bc=1$. This proves the claim. Next, one can show that $pz_{i}$ is the zero or pole of $f|\T(p)$. It implies that the following:
\begin{equation*}
\bigcup\limits_{i=1}^{n}\{pz_{i}:p\text{ is a prime}\}/\sim_{\G_{0}(N)}\subset\{z_{1},\cdots,z_{n}\}.
\end{equation*}
Since the number of an elements in the first set is infinite, we obtain a contradiction.
\end{proof}

\section{Hecke equivariance}

First, we prove Theorem \ref{main1} and Theorem \ref{Dequivn} for the power of primes, and then extend them to all positive integers $n$ using Theorem \ref{Heckealgebra}.

\begin{thm}\label{Bequiv}
Let $p$ be a prime not dividing $N$ nor the discriminant $\Delta$ and $r$ be a positive integer. Let $H_{k,\tilde{\rho}_{N}}^{'}$ be the additive subgroup of $H_{k,\tilde{\rho}_{N}}$ consisting of forms that satisfy the conditions in \cite[Theorem 6.1]{BO}. Then the following diagrams of the groups are commutative.
\begin{equation*}
\begin{tikzcd}
H_{\frac{1}{2},\tilde{\rho}_{N}}^{'}\arrow{r}{B}\arrow{d}{p^{r}T_{\frac{1}{2}}(p^{2r})} & \M^{H}(N)\arrow{d}{\T(p^{r})}
\\
H_{\frac{1}{2},\tilde{\rho}_{N}}^{'}\arrow{r}{B} & \M^{H}(N)
\end{tikzcd}
\end{equation*}
\end{thm}

\begin{proof}
We use induction on $r$. When $r=1$, it was proven in \cite[Theorem 3.1]{JKK}. We assume that this theorem holds for $r<k$. Let $f\in H_{\frac{1}{2},\tilde{\rho}_{N}}^{'}$. Then, we have
\begin{align*}
B(f)|\T(p^{k})=\frac{B(f)|\T(p^{k-1})\T(p)}{\big(B(f)|\T(p^{k-2})\big)^{p}}=\frac{B\big(f|p^{k}T_{\frac{1}{2}}(p^{2k-2})T_{\frac{1}{2}}(p^{2})\big)}{(B(f|p^{k-2}T_{\frac{1}{2}}(p^{2k-4}))\big)^{p}}
\\
=B\big(f|p^{k}T_{\frac{1}{2}}(p^{2k-2})T_{\frac{1}{2}}(p^{2})-p^{k-1}T_{\frac{1}{2}}(p^{2k-4})\big)=B(f|p^{k}T_{\frac{1}{2}}(p^{2k})).
\end{align*}
The first equality follows from \eqref{relation}. The second equality follows from the induction hypothesis. As the Borcherds product is a homomorphism, the third equation is obtained. The last equality follows from the formula for half-integral weight additive Hecke operators.
\end{proof}

Next, we prove Theorem \ref{Dequivn} for the power of primes such that $(p,N)=1$.

\begin{thm}\label{commd}
Let $N$ be a positive integer and $p$ be a prime such that $(p,N)=1$. Let $r$ be a positive integer. Then the following diagram is commutative:
\begin{equation}
\begin{tikzcd}
\M(N)\arrow{r}{\Dif}\arrow{d}{\T(p^{r})} & M_{2}^{\text{mero}}(N)\arrow{d}{T_{2}(p^{r})}
\\
\M(N)\arrow{r}{\Dif} & M_{2}^{\text{mero}}(N)
\end{tikzcd}
\end{equation}
\end{thm}

\begin{proof}
One can easily see that
\begin{equation*}
T_{2}(p^{r})E_{2}(\t)=\sigma(p^{r})E_{2}(\t).
\end{equation*}
Thus it suffices to show that 
\begin{equation}\label{comm}
\frac{\Theta(f|\T(p^{r}))}{f|\T(p^{r})}=\frac{\Theta(f)}{f}|T_{2}(p^{r}).
\end{equation}
Let 
\begin{equation*}
f(\t)=q^{h}\prod_{n=1}^{\i}(1-q^{n})^{c(n)}\in\M_{k,h}(N).
\end{equation*}
Then, by \cite[Proposition 2.1]{BKO}, we have
\begin{equation*}
\frac{\Theta(f)}{f}=h-\sum\limits_{n=1}^{\i}\bigg(\sum\limits_{d\mid n}c(d)d\bigg)q^{n}=:\sum\limits_{n=1}^{\i}a(n)q^{n}.
\end{equation*}
Moreover, we have
\begin{equation*}
\frac{\Theta(f)}{f}|T_{2}(p^{r})=\sum\limits b(n)q^{n},
\end{equation*}
where
\begin{equation*}
b(n)=
\begin{cases}
\sigma(p^{r})h & \text{ if } n=0
\\
-\sum\limits_{d\mid p^{m}}da(\frac{np^{r}}{d^{2}}) & \text{ if } (n,p^{r})=p^{m}(0\leq m\leq r).
\end{cases}
\end{equation*}
Next, we have.
\begin{align*}
\frac{\Theta(f|\T(p^{r}))}{f|\T(p^{r})}&=\frac{1}{2\pi i}\frac{d}{d\t}\log\bigg(q^{h(\frac{p^{r+1}-1}{p-1})}\prod_{n=1}^{\i}(1-q^{n})^{c_{p^{r}}(n)}\bigg)
\\
&=h\bigg(\frac{p^{r+1}-1}{p-1}\bigg)-\sum\limits_{n=1}^{\i}nc_{p^{r}}(n)\frac{q^{n}}{1-q^{n}}=h\bigg(\frac{p^{r+1}-1}{p-1}\bigg)-\sum\limits_{n=1}^{\i}nc_{p^{r}}(n)\sum\limits_{m=1}^{\i}q^{mn}
\\
&=h\bigg(\frac{p^{r+1}-1}{p-1}\bigg)-\sum\limits_{n=1}^{\i}\sum\limits_{d\mid n}\Big(dc_{p^{r}}(d)\Big)q^{n}=:\sum\limits C(n)q^{n}
\end{align*}
It is clear that the constant terms on both sides of \eqref{comm} are the same. Next, we suppose that $(n,p^{r})=p^{m}$ where $0\leq m\leq r$ and write $n=p^{m}e$ with some positive integer $e$. In this case, $C(n)$ is expressed as
\begin{align*}
C(n)&=-\sum\limits_{d|n}d\bigg(\sum\limits_{i=0}^{r}p^{i}\c\bigg(p^{i},\frac{d}{p^{r-i}}\bigg)+\sum\limits_{k=0}^{r-1}\sum\limits_{i=0}^{k}\chi_{p}\bigg(\frac{d}{p^{r-k-1}}\bigg)p^{i}\c\bigg(p^{i},\frac{d}{p^{r-k-1}}\bigg)\bigg)
\\
&=-\sum\limits_{j=0}^{m}\sum\limits_{l|e}p^{j}l\bigg(\sum\limits_{i=r-j}^{r}p^{i}\c\bigg(p^{i},\frac{l}{p^{r-i-j}}\bigg)+\sum\limits_{i=0}^{r-j-1}p^{i}\c(p^{i},l)\bigg).
\end{align*}
Note that 
\begin{equation*}
b(n)=-\sum\limits_{d\mid p^{m}}\sum\limits_{d'|\frac{np^{r}}{d^{2}}}dd'c(d')=-\sum\limits_{j=0}^{m}\sum\limits_{l|e}\sum\limits_{i=0}^{m+r-2j}p^{i+j}lc(p^{i}l).
\end{equation*}
Thus, it suffices to show that
\begin{equation}\label{change}
\sum\limits_{j=0}^{m}\sum\limits_{i=r-j}^{r}p^{i+j}c(p^{2i+j-r}l)=\sum\limits_{j=0}^{m}\sum\limits_{i=r-j}^{m+r-2j}p^{i+j}c(p^{i}l)
\end{equation}
for all $l|e$. It can be easily verified by changing the variables  $u:=2i+j-r$ and $u+v:=i+j$ on the right-hand side in \eqref{change}.
\end{proof}

\begin{proof}[Proof of Theorem \ref{main1} and Theorem \ref{Dequivn}]
We write that $n=\prod p_{i}^{r_{i}}$. From Theorem \ref{commd}, we have 
\begin{align*}
\Dif\T(n)&=\Dif\T(p_{1}^{r_{1}})\cdots\T(p_{m}^{r_{m}})=T_{2}(p_{1}^{r_{1}})\Dif\T(p_{2}^{r_{2}})\cdots\T(p_{m}^{r_{m}})=
\\
&\cdots =T_{2}(p_{1}^{r_{1}})\cdots T_{2}(p_{m}^{r_{m}})\Dif=T_{2}(n)\Dif.
\end{align*}
This is the proof of Theorem \ref{Dequivn}. The proof of Theorem \ref{main1} follows immediately from Theorem \ref{Bequiv} and the above argument.
\end{proof}

\end{document}